\documentclass[12pt, reqno, letterpaper]{amsart}
\usepackage{enumitem}

\numberwithin{equation}{section}

\usepackage[bindingoffset=0.2in,%
            left=1in,right=1in,top=1in,bottom=1in,%
            footskip=.25in]{geometry}
\newtheorem*{claim}{Claim}
\newtheorem{theorem}{Theorem}[section]
\newtheorem{definition}[theorem]{Definition}
\newtheorem{proposition}[theorem]{Proposition}
\newtheorem{lemma}[theorem]{Lemma}
\newtheorem{corollary}[theorem]{Corollary}
\newtheorem{remark}[theorem]{Remark}

\def\al{\aligned}
\def\eal{\endaligned}
\def\be{\begin{equation}}
\def\ee{\end{equation}}
\def\lab{\label}
\def\a{\alpha}

\def\e{\epsilon}

\def\M{{\bf M}}

\def\al{\aligned}

\def\pa{\partial}
\def\d{\nabla}

\def\lam{\lambda}

\def\v{\textrm{Vol}}
\def\mb{\mathbb}
\def\mc{\mathcal}

\DeclareMathOperator{\Ric}{Ric}

\DeclareMathOperator{\vol}{Vol}

\DeclareMathOperator{\Hess}{Hess}

\usepackage{color}

\numberwithin{equation}{section}

\let\oldtocsection=\tocsection

\let\oldtocsubsection=\tocsubsection

\let\oldtocsubsubsection=\tocsubsubsection

\renewcommand{\tocsection}[2]{\hspace{0em}\oldtocsection{#1}{#2}}
\renewcommand{\tocsubsection}[2]{\hspace{2em}\oldtocsubsection{#1}{#2}}
\renewcommand{\tocsubsubsection}[2]{\hspace{2em}\oldtocsubsubsection{#1}{#2}}

\begin{document}

\title[]{New Volume Comparison results and Applications to degeneration of
Riemannian metrics }
\author{Qi S. Zhang and Meng Zhu}
\address{Department of Mathematics, University of California, Riverside, Riverside, CA 92521, USA}
\email{qizhang@math.ucr.edu}
\address{Department of Mathematics, East China Normal University, Shanghai 200241, China - and Department of Mathematics, University of California, Riverside, Riverside, CA 92521, USA}
\email{mzhu@math.ucr.edu}
\date{}

\begin{abstract}
We consider a condition on the Ricci curvature involving vector fields, which is broader than the Bakry-\'Emery Ricci condition. Under this condition
volume comparison, Laplacian comparison, isoperimetric inequality and gradient bounds
are proven on the manifold.

Specializing to the Bakry-\'Emery Ricci curvature condition,  we initiate an approach to work on the original manifold,  which yields, under a weaker than usual assumption,
the results mentioned above
 for the {\it original manifold}. These results are different from most well known ones in the literature where the conclusions are made on the weighted manifold instead.

Applications on convergence and degeneration of Riemannian metrics under this curvature condition are given. To this effect, in particular for the Bakry-\'Emery Ricci curvature condition, the gradient of the potential function is allowed to have singularity of order
close to $1$ while the traditional method of weighted manifolds allows bounded gradient.
This approach enables us to extend some of the results in the papers \cite{Co}, \cite{ChCo2}, \cite{zZh}, \cite{TZ} and \cite{WZ}.
The condition also covers general Ricci solitons instead of just gradient Ricci solitons.
\end{abstract}
\maketitle

\tableofcontents

\section{Introduction}

The volume comparison theorem by Bishop-Gromov is an essential tool in differential geometry and
analysis on manifolds. It states that if the Ricci curvature is bounded from below by a constant, the ratio  between the volume of a geodesic ball and the volume of the ball of the same radius in the model space is a non-increasing function of the radius.
However, in many situations, the point-wise lower bound on the Ricci curvature is not available. Hence many efforts are made to relax this assumption. In the paper \cite{PeWe1}, Petersen and Wei obtained a bound on the above volume ratio under the assumption that the norm of the
negative part of the Ricci curvature is an $L^p$ function with $p>n/2$, where $n$ is the dimension of the manifold. The bound approaches $1$ when the radii of the balls converge to $0$.  Another successful approach of relaxing the Ricci lower bound condition is
 to consider a
lower bound of the Bakry-\'Emery Ricci curvature  (\cite{BE}).  A typical condition is
\be\label{BERicci}
\Ric+\textrm{Hess} L\geq - \lam g
\ee for a nonnegative constant $\lam$ and a smooth function $L$ such that $L$ and/or $|\nabla L|$ is bounded.
Under this condition, for the weighted manifold $(\M,   dvol= e^{-L} dg)$,  several authors derived Laplacian comparison theorems for the weighted Laplacian $\Delta - \nabla L \cdot\nabla$  and/or bounds on the
volume ratio of weighted balls.  See Qian \cite{Q}, Bakry and Qian \cite{BQ},  X.D. Li  \cite{xLi} and Wei and Wylie \cite{WW}.  All of these results have led to many applications which can be found in the forward references in MathSciNet etc.
Their results also yield certain bounds on volumes of geodesic balls on the original manifold involving $\sup L$ and $\inf L$, or $\nabla L$. However,
monotonicity or almost monotonicity of volume ratios on the original geodesic balls do not follow directly in general.

In this paper,  we consider  a Ricci curvature condition involving vector fields, which is broader than the Bakry-\'Emery Ricci curvature
condition \eqref{BERicci}. It is also more general than curvature dimension inequality.  More specifically, we impose a lower bound condition on a modified Ricci curvature
\be
\lab{ric.vec}
\Ric+\frac{1}{2}\mc{L}_{V} g
\ee where $V$ is a smooth vector field, and $\mc{L}_Vg$ is the Lie derivative
of $g$ in the direction of $V$.  In local coordinates, \eqref{ric.vec} can be written as
\[
R_{ij} + \frac{1}{2} (\nabla_i V_j  + \nabla_j V_i),
\] where $V=V^i \partial_i$ and $V_i=g_{ij} V^j$.
If $V= \nabla L$ for a smooth function $L$, then this is nothing but the Bakry-\'Emery Ricci curvature. The study of Ricci solitons and Einstein field equation provide strong motivation for considering this kind of Ricci condition. Let us recall that a Riemannian manifold $({\M}, g)$ is called a Ricci soliton if there exists
a smooth vector field $V$ and a constant $\lambda$ such that
\[
\Ric+\frac{1}{2}\mc{L}_{V} g + \lambda g=0.
\] Also Einstein's field equation can be expressed as
\be
\lab{EFE}
R_{ij} - \frac{1}{2} R g_{ij} + \lambda g_{ij}
=  \frac{ 8 \pi G}{c^4} T_{ij}
\ee where $R$ is the scalar curvature, $G$ is Newton's gravitational
constant and $c$ is the speed of light in vacuum.
$T_{ij}$ is the stress-energy tensor which may take the special form
$\nabla_i V_j  + \nabla_j V_i$  or $\nabla^2 L$.  Another motivation comes from the study of the so called modified Ricci
flows where the Ricci curvature is modified by a vector field in the same manner as here.
In addition, this kind of conditions have been used in other context such as
\cite{D},
\cite{FG} and \cite{W} to obtain diameter bound or bounds on fundamental groups
for manifolds .

An underlining feature of the paper is that our emphasis is on the original manifolds and the
standard Laplacian. By assuming the modified Ricci curvature \eqref{ric.vec} having a lower bound,
volume comparison, Laplacian comparison and gradient bounds are proven for the manifold. This approach has the advantage that even for the Bakry-\'Emery Ricci curvature, i.e, $V= \nabla L$, it can be shown that under a weaker than usual condition on the potential function that much of the volume ratio bounds, Laplacian comparison theorem and gradient bounds still hold for the {\it original manifolds}. For instance, the results here provide almost monotonicity of volume ratios in the spirit of Bishop-Gromov volume comparison on the original manifold without the boundedness of the $|\nabla L|$. These properties are necessary for proving convergence results on general manifolds, Ricci solitons and other applications which will be presented later in the paper. Here are samples of the main theorems.

\begin{theorem}[Tangent cones are metric cones, Theorem \ref{thtcmc} below]
Let $(\M_i, g_i)$ be a sequence of Riemannian manifolds satisfying \eqref{basichypo1} and $\eqref{basichypo2}$ for some
$\lam \ge 0,\ K \ge 0,\ \a \in [0, 1), $ and $\rho>0$ and some smooth vector field $V_i$, i.e.,
\[
\Ric_{ g_i} +\frac{1}{2}\mc{L}_{V_i} g_i\geq - \lam g_i; \quad
|V_i|(\cdot)\leq \frac{K}{d_i(\cdot, O_i)^\a}, \ \text{for some} \ O_i \in \M_i;\
 \vol (B_{ g_i}(\cdot, 1)) \ge \rho.
\]
Suppose that $(\M_i,d_i)\xrightarrow{d_{GH}}(Y,d)$ in Gromov-Hausdorff topology. Then any tangent cone of $Y$ is a metric cone.
In particular, any tangent cone of a single manifold satisfying the above condition is a metric cone.
\end{theorem}

\begin{theorem}[Size estimate of the singular set, Theorem \ref{thsizesc} below]
Let $(\M_i, g_i)$ be a sequence of Riemannian manifolds satisfying \eqref{basichypo1} and \eqref{basichypo2} for some $\lam,\ K,\ \a,$ and $\rho$. Suppose that $(\M_i,d_i)\xrightarrow{d_{GH}}(Y,d)$ in Gromov-Hausdorff topology. Then $dim S(Y) \le n-2$. Here, $S(Y)$ is the singular set of $Y$ and $dim$ stands for the Hausdorff dimension.
\end{theorem}

The following is an outline of the proof.  We first prove the volume and Laplacian comparison results under the Ricci vector field condition. Then we will move according to the road map provided in the papers \cite{ChCo1} and \cite{ChCo2}. Namely, we need to take  a number of intermediate steps of proving gradient bounds for harmonic functions, Sobolev inequality, existence of good cut off functions, segment inequality, excess estimate, barrier functions and existence of $\e$-splitting maps. With the exception of gradient bound and segment inequality, the proofs
of these results will involve new ingredients drawing from recent literature, which
 are needed since many of the model functions used in comparison arguments are
not available here.  The main assumption on the vector field $V$ in this paper is $|V(\cdot)| \le \frac{K}{d(\cdot, O)^\a}$ , for some $O \in \M$ and $\a \in [0, 1)$. If $\a=0$, i.e., the vector field $V$ is bounded, then the proof can be shortened considerably. Some results
can also be strengthened.

\section{New volume comparison theorems}

\subsection{Ricci vector field condition}

Let $(\M, g)$ be an $n$ dimensional Riemannian manifold with a  fixed base point $O\in\M$. In this subsection,
our basic assumption on the Ricci curvature tensor is
\be\label{riccond}
\Ric+\frac{1}{2}\mc{L}_{V} g\geq - \lam g
\ee
for some constant $\lam \geq0$ and smooth vector field $V$ which satisfies the following condition:
\be\label{condition V}
|V|(y)\leq \frac{K}{d(y,O)^\a}
\ee
for any $y\in\M$.  Here $d(y,O)$ represents the distance from $O$ to $y$, and $K\geq0$ and $0\leq\a<1$ are constants. Intuitively, the basic assumption \eqref{condition V} on the vector field $V$ is that it has singularity of order less than $1$.  Our assumptions include the popular special cases such as \eqref{BERicci}, i.e., Bakry-\'Emery Ricci curvature $\Ric+\Hess L$ being bounded from below, with bounded $|\nabla L|$, and gradient Ricci solitons $\Ric +\Hess L=\lam g$ with just bounded potential $L$. In both of these special cases, we have $V=\d L$.  Moreover, due to the special structure of gradient Ricci solitons, local boundedness on the potential function $L$ automatically implies local bounds on its gradient.

Let us present the Laplacian comparison theorem on which the paper is based.  The
proof replies on a simple but key observation that the Lie derivative or Hessian become
ordinary derivative or second derivative in the radial direction.

We should mention a related result by Qian \cite{Q} who long time ago obtained a Laplacian comparison theorem under a Ricci curvature condition that is modified
 by a vector field:
 \[
 \Ric+\frac{1}{2}\mc{L}_{V} g -\frac{1}{N-n} V \otimes V \ge - \lam g.
 \]Here $V$ is a vector field, $n$ is the dimension of the manifold and $N$ is a number strictly greater than $n$.
 However, the result depends on $N$ which can not be  taken  as $n$, and the Laplace operator there is also modified accordingly. In the following, we will use $C(a_1,a_2,\cdots, a_k)$ to denote constants depending on parameters $a_1, a_2,\cdots, a_k$.

\begin{proposition}[Laplacian comparison]\label{prop laplacian comparison}
Assume that \eqref{riccond} and \eqref{condition V} hold. Let $s=d(y,x)$ be the distance from any point $y$ to some fixed point $x$,
and $\gamma:\ [0,s]\rightarrow\M$ a normal minimal geodesic with
 $\gamma(0)=x$ and $\gamma(s)=y$.
Then in the distribution sense,
\be\label{Laplace comparison}
\Delta s-\frac{n-1}{s}\leq \frac{\lam}{3}s+<V, \d s>+\frac{C(\a)K}{s^\a}.
\ee
\end{proposition}

\proof  We may just prove the inequality at smooth points of $s$, since \eqref{Laplace comparison} can then be established in the distribution sense by the fact that the complement of the cut locus is star shaped.

From the Bochner formula, we have
\[
0=\frac{1}{2}\Delta|\d s|^2=|\d^2 s|^2 + <\d\Delta s, \d s> + \Ric(\d s, \d s).
\]
By using Cauchy-Schwarz inequality $|\d^2 s|^2\geq\frac{(\Delta s)^2}{n-1}$, it yields
\[
\frac{\pa}{\pa s}\Delta s+\frac{(\Delta s)^2}{n-1}\leq -\Ric(\d s, \d s),
\]
which is equivalent to
\be\label{d psi}
\frac{1}{s^2}\frac{\pa}{\pa s}(s^2\Delta s)+\frac{1}{n-1}(\Delta s-\frac{n-1}{s})^2\leq \frac{n-1}{s^2}- \Ric(\d s, \d s),
\ee
due to the fact that
\[
\al
\frac{\pa}{\pa s}\Delta s+\frac{(\Delta s)^2}{n-1}&=\frac{\pa}{\pa s}\Delta s+\frac{(\Delta s)^2}{n-1}-\frac{2}{s}\Delta s + \frac{n-1}{s^2} + \frac{2}{s}\Delta s - \frac{n-1}{s^2}\\
&=\frac{1}{s^2}\frac{\pa}{\pa s}(s^2\Delta s) - \frac{n-1}{s^2}+\frac{1}{n-1}(\Delta s-\frac{n-1}{s})^2.
\eal
\]
Multiplying both sides of \eqref{d psi} by $s^2$ and integrating from $0$ to $s$, we get
\be\label{Delta s}
\Delta s\leq \frac{n-1}{s}-\frac{1}{s^2}\int_0^st^2\Ric(\gamma'(t), \gamma'(t))dt.
\ee

At any point $\gamma(t)$,
we may choose an orthonormal frame $\{e_1,e_2,\cdots,e_n\}$
 with $e_1=\gamma'(t)$. Then from assumption \eqref{riccond}, we observe that
\[
\al
\Ric(\gamma'(t),\gamma'(t))&=\Ric(e_1,e_1)\\
&\geq-\lam-\frac{1}{2}\mc{L}_V g(e_1,e_1)\\
&=-\lam-<\d_{e_1}V, e_1>\\
&=-\lam-e_1<V,e_1> + <V, \nabla_{e_1} e_1>\\
&=-\lam-\frac{\pa}{\pa t}<V, \gamma'(t)>.
\eal
\] This is the key simplification that allows us to proceed.
Hence, \eqref{Delta s} becomes
\be\label{Delta s 1}
\al
\Delta s- \frac{n-1}{s}&\leq\frac{1}{s^2}\int_0^s \left[t^2\frac{\pa}{\pa t}<V, \gamma'(t)>+\lam t^2\right]dt\\
&\leq \frac{\lam}{3}s+\frac{1}{s^2}\left[t^2<V,\gamma'(t)>(\gamma(t))\big|_0^s-2\int_0^st<V,\gamma'(t)>dt\right]\\
&\leq \frac{\lam}{3}s+<V, \d s> -\frac{2}{s^2}\int_0^s t<V,\gamma'(t)>dt.
\eal
\ee
Denote $d_0=d(x,O)$.

\hspace{-.5cm}{\it Case 1:} If $s\leq d_0$, then by the triangle inequality, we have
\[
d(\gamma(t),O)\geq d_0-t.
\]
Hence,
\be\label{eq1}
\al
-\frac{2}{s^2}\int_0^s t<V,\gamma'(t)>dt&\leq \frac{1}{s^2}\int_0^s 2t\cdot\frac{K}{(d_0-t)^{\a}}dt\\
&\leq \frac{C(\a)K}{s}[-(d_0-t)^{1-\a}\big|_0^s]\\
&=\frac{C(\a)K}{s}[d_0^{1-\a}-(d_0-s)^{1-\a}]\\
&\leq\frac{C(\a)K}{s^{\a}}.
\eal
\ee
In the last step above, we have used the fact that
\be\label{Jensen}
s^{1-\a}+(d_0-s)^{1-\a}\geq d_0^{1-\a}.
\ee
Indeed, the above inequality can be obtained by setting $x=\frac{s}{d_0-s}$ in the inequality
\be\label{Cauchy-Schwarz}
x^{1-\a}+1 \geq(1+x)^{1-\a},\quad x>0.
\ee
By plugging \eqref{eq1} into \eqref{Delta s 1}, one has
\be\label{Laplacian comparison<d0}
\Delta s-\frac{n-1}{s}\leq \frac{\lam}{3}s+<V, \d s>+ \frac{C(\a)K}{s^{\a}}.
\ee

\hspace{-.5cm}{\it Case 2:} If $s>d_0$, then for any $d_0\leq t\leq s$, we have
$$d(\gamma(t),O)\geq t-d_0,$$
and
\be\label{eq2}
\al
-\frac{2}{s^2}\int_0^s t<V,\gamma'(t)>dt&\leq \frac{1}{s^2}\int_0^{d_0} 2t\cdot\frac{K}{(d_0-t)^{\a}}dt+\frac{1}{s^2}\int_{d_0}^s2t\cdot\frac{K}{(t-d_0)^{\a}}dt\\
&\leq \frac{C(\a)Kd_0}{s^2}d_0^{1-\a}+\frac{C(\a)K(s-d_0)^{1-\a}}{s}\\
&=\frac{C(\a)K}{s^{\a}}.
\eal
\ee

Therefore, from \eqref{eq2}, we deduce
\be\label{Laplacian comparison>d0}
\Delta s-\frac{n-1}{s}\leq \frac{\lam}{3}s+<V, \d s> + \frac{C(\a)K}{s^{\a}}.
\ee
This finishes the proof of the proposition. \qed \\

Here is the main result of this subsection.

\begin{theorem}[Volume comparison]\label{thm volume element comparison}
Assume that \eqref{riccond} and \eqref{condition V} hold.

(a). Let $w(s,\cdot)$ denote the volume element of the metric $g$ on $\M$ in geodesic polar coordinates. Then for any $0<s_1<s_2$, we have
\be\label{volume element ratio}
\frac{w(s_2,\cdot)}{s_2^{n-1}}\leq e^{C(\a)K s_2^{1-\a}+\lam s_2^2}\frac{w(s_1,\cdot)}{s_1^{n-1}}.
\ee
In particular, by letting $s_1\rightarrow 0$, we get
\be\label{volume element}
w(s,\cdot)\leq e^{C(\a)K s^{1-\a}+\lam s^2}s^{n-1},\quad \forall s>0,
\ee
and hence
\be\label{volume noninflation}
\vol(B(x,r))\leq e^{C(\a)K r^{1-\a}+\lam r^2}\frac{\vol(S^{n-1})}{n}r^n,\quad \forall r>0.
\ee
Here, $S^{n-1}$ is the unit sphere in $\mathbb{R}^n$.

(b). Suppose in addition that the volume non-collapsing condition holds
\be\label{noncollapsing-0}
\vol(B(x,1))\geq \rho,
\ee
for all $x\in\M$ and some constant $\rho>0$.
  Then for any $0<r_1<r_2\leq 1$, we have the volume ratio bound
\be\label{volume comparison}
\al
\frac{\vol(B(x,r_2))}{r_2^n}\leq e^{C(n,\lam,K,\a,\rho)[\lam(r_2^2-r_1^2)+K(r_2-r_1)^{1-\a}]}\cdot\frac{\vol(B(x,r_1))}{r_1^n}.\\
\eal
\ee

(c).  In particular, suppose that the Bakry-\'Emery Ricci curvature tensor $Ric+\Hess L$ satisfies
\be
Ric+\Hess L\geq - \lambda g,
\ee
and
\be\label{condition V=dL}
|\nabla L(y)| \le \frac{K}{d(y, O)^\a}
\ee
for all $y \in \M$, a fixed point $O \in \M$,
and constants $\lam\geq 0$, $K\geq0$, $\a \in [0, 1)$, and assume in addition that
 the noncollapsing condition (\ref{noncollapsing-0}) holds.
Then the conclusions of part (a) and (b) are true.

\end{theorem}

\proof[Proof of part (a).]
 First of all, from \eqref{Laplace comparison} we have
\be\label{differential lnw}
\frac{\pa}{\pa s}\ln(w(s,\cdot))=\frac{w'(s,\cdot)}{w(s,\cdot)}=\Delta s\leq \frac{n-1}{s}+\frac{\lam}{3}s+<V, \d s> +\frac{C(\a)K}{s^\a}.
\ee

Recall that $d_0=d(x,O)$.

\hspace{-.5cm}{\it Case 1:} If $s_1<s_2\leq d_0$, by using the assumption \eqref{condition V} and triangle inequality, \eqref{differential lnw} can be transformed as
\[
\frac{\pa}{\pa s}\ln(w(s,\cdot))\leq \frac{n-1}{s}+\frac{\lam}{3}s+\frac{K}{(d_0-s)^\a}+\frac{C(\a)K}{s^\a}.
\]
Integrating both sides from $s_1$ to $s_2$ gives
\be\label{w comparison 1}
\al
\ln\frac{w(s_2,\cdot)}{w(s_1,\cdot)}&\leq \ln(\frac{s_2}{s_1})^{n-1}+\frac{\lam}{6}(s_2^2-s_1^2)+C(\a)K(s_2^{1-\a}-s_1^{1-\a})\\
&\quad +C(\a)K[(d_0-s_1)^{1-\a}-(d_0-s_2)^{1-\a}]\\
&\leq \ln(\frac{s_2}{s_1})^{n-1}+\frac{\lam}{6}s_2^2+C(\a)Ks_2^{1-\a}+C(\a)K(s_2-s_1)^{1-\a}\\
&\leq \ln(\frac{s_2}{s_1})^{n-1}+\frac{\lam}{6}s_2^2+C(\a)Ks_2^{1-\a},
\eal
\ee
where in the second last step we have used \eqref{Cauchy-Schwarz} for $x=\frac{s_2-s_1}{d_0-s_2}$ to derive
\[
(d_0-s_1)^{1-\a}-(d_0-s_2)^{1-\a}\leq (s_2-s_1)^{1-\a}.
\]
In particular, one has for $s_1\leq d_0$,
\be\label{w comparison 1a}
\ln\frac{w(d_0,\cdot)}{w(s_1,\cdot)}\leq \ln(\frac{d_0}{s_1})^{n-1}+\frac{\lam}{6}d_0^2+C(\a)Kd_0^{1-\a}.
\ee
\hspace{-.5cm}{\it Case 2:} If $d_0\leq s_1<s_2$, similarly since $$<V, \d s> \leq |V|(\gamma(s))\leq\frac{K}{(s-d_0)^\a},$$
we deduce from \eqref{differential lnw} that
\be\label{w comparison 2}
\al
\ln\frac{w(s_2,\cdot)}{w(s_1,\cdot)}&\leq \ln(\frac{s_2}{s_1})^{n-1}+\frac{\lam}{6}(s_2^2-s_1^2)+C(\a)K(s_2^{1-\a}-s_1^{1-\a})\\
&\quad +C(\a)K[(s_2-d_0)^{1-\a}-(s_1-d_0)^{1-\a}]\\
&\leq \ln(\frac{s_2}{s_1})^{n-1}+\frac{\lam}{6}s_2^2+C(\a)Ks_2^{1-\a}+C(\a)K(s_2-s_1)^{1-\a}\\
&\leq \ln(\frac{s_2}{s_1})^{n-1}+\frac{\lam}{6}s_2^2+C(\a)Ks_2^{1-\a}.
\eal
\ee
Especially, it implies that for $s_2\geq d_0$,
\be\label{w comparison 2a}
\ln\frac{w(s_2,\cdot)}{w(d_0,\cdot)}\leq \ln(\frac{s_2}{d_0})^{n-1}+\frac{\lam}{6}s_2^2+C(\a)Ks_2^{1-\a}.
\ee

\hspace{-.5cm}{\it Case 3:} If $s_1\leq d_0\leq s_2$, then by adding \eqref{w comparison 1a} to
\eqref{w comparison 2a}, we get
\be\label{w comparison 3}
\ln\frac{w(s_2,\cdot)}{w(s_1,\cdot)}\leq \ln(\frac{s_2}{s_1})^{n-1}+\frac{\lam}{6}s_2^2+C(\a)Ks_2^{1-\a}.
\ee
Now \eqref{volume element ratio} is a direct consequence of
\eqref{w comparison 1}, \eqref{w comparison 2}, and \eqref{w comparison 3}.
This proves part (a) of the theorem.\\

In order to prove part (b) of the theorem, we need two more elementary results besides Proposition \ref{prop laplacian comparison}.

First, we present a result which is essentially known
in \cite{PeWe1}. We give a proof for completeness.
\begin{lemma}\label{proposition differential of volume ratio}
Let $B(x,r)$ denote the geodesic ball in $\M$ centered at $x$ with radius $r$, and $\vol(B(x,r))$ denote the volume of $B(x,r)$. Then it holds that
\be\label{differential of volume ratio}
\frac{d}{dr}\left(\frac{\vol(B(x,r))}{r^n}\right)\leq \frac{1}{r^{n+1}}\int_{S^{n-1}}\int_0^r s\,\psi\, w(s,\theta)dsd\theta,
\ee
where $\psi=(\Delta s-\frac{n-1}{s})_+=\max\{\Delta s-\frac{n-1}{s}, 0\}$, $S^{n-1}$ denotes the unit sphere in $\mb{R}^n$ and $d\theta$ is its volume element.
\end{lemma}
\begin{proof} Notice that
\[
\frac{d}{d r}\v(B(x,r))=\frac{d}{dr}\int_0^r\int_{S^{n-1}}w(t,\theta)d\theta dt=\int_{S^{n-1}}w(r,\cdot)d\theta.
\]
Therefore,
\[
\al
\frac{d}{dr}\left(\frac{\v(B(x,r))}{r^n}\right)&=\frac{r^n\int_{S^{n-1}}w(r,\theta)d\theta-nr^{n-1}\v(B(x,r))}{r^{2n}}\\
&=\frac{n\int_{S^{n-1}}w(r,\theta)d\theta\int_0^rt^{n-1}dt-nr^{n-1}\int_0^r\int_{S^{n-1}}w(t,\theta)d\theta dt}{r^{2n}}\\
&=\frac{n}{r^{2n}}\int_{S^{n-1}}\int_0^r\left[t^{n-1}w(r,\theta)-r^{n-1}w(t,\theta)\right]dtd\theta\\
&=\frac{n}{r^{2n}}\int_{S^{n-1}}\int_0^r t^{n-1}r^{n-1}\left[\frac{w(r,\theta)}{r^{n-1}}-\frac{w(t,\theta)}{t^{n-1}}\right]dtd\theta\\
&=\frac{n}{r^{n+1}}\int_{S^{n-1}}\int_0^r t^{n-1}\left(\int_t^r\frac{d}{ds}\left[\frac{w(s,\theta)}{s^{n-1}}\right]ds\right)dtd\theta\\
&=\frac{n}{r^{n+1}}\int_{S^{n-1}}\int_0^r t^{n-1}\left(\int_t^r \left[\frac{sw'(s,\cdot)-(n-1)w(s,\cdot)}{s^{n}}\right]ds\right)dtd\theta\\
&\leq \frac{n}{r^{n+1}}\int_{S^{n-1}}\left(\int_0^r\int_t^rt^{n-1}\frac{\psi w(s,\theta)}{s^{n-1}}dsdt\right)d\theta \quad
(\text{by} \  w'= w \Delta s),\\
&=\frac{n}{r^{n+1}}\int_{S^{n-1}}\left(\int_0^r\int_0^st^{n-1}\frac{\psi w(s,\theta)}{s^{n-1}}dtds\right)d\theta\\
&=\frac{1}{r^{n+1}}\int_{S^{n-1}}\int_0^r s\,\psi\, w(s,\theta)dsd\theta.
\eal
\]
\end{proof}

We will also need the following

\begin{lemma}\label{lemma Lq bound of V}
Suppose that in addition to the assumptions \eqref{riccond} and \eqref{condition V}, the following volume noncollapsing condition holds
\be\label{noncollapsing}
\vol(B(x,1))\geq \rho,
\ee
for all $x\in\M$ and some constant $\rho>0$.
Then for any $q\in(0,\frac{n}{\a})$ we have
\be\label{Lq bound of V}
\sup_{x\in\M,0< r\leq 1}r^{\a}||V||^*_{q,B(x,r)}\leq C(n,K,\lam,\a,\rho,q)K,
\ee
where
\[
||V||^*_{q,B(x,r)}=\left(\oint_{B(x,r)}|V|^q dg\right)^{1/q}=\left(\frac{1}{\vol(B(x,r))}\int_{B(x,r)}|V|^q dg\right)^{1/q}.
\]
\end{lemma}

\proof For any $0<r\leq 1$, according to the assumption \eqref{condition V},
\be\label{case 0}
\al
||V||^*_{q,B(x,r)}&\leq \left(\frac{1}{\v(B(x,r))}\int_{B(x,r)}\frac{K^q}{d(y,O)^{\a q}} dg(y)\right)^{1/q}.
\eal
\ee

\hspace{-.5cm}{\it Case 1:} If $d(x,O)\leq 2r$, then $B(x,r)\subset B(O,3r)$, then for any $0<q<\frac{n}{\a}$, we have
\be\label{case 1}
\int_{B(x,r)}\frac{1}{d(y,O)^{\a q}} dg(y)\leq \int_{B(O,3r)}\frac{1}{d(y,O)^{\a q}}dg(y)
\leq C(n,\a,q)e^{C(\a)Kr^{1-\a}+\lam r^2}r^{n-\a q}.
\ee
Indeed,  by  \eqref{volume element},  for any $\gamma<n$,  the following holds,
\be\label{integral of distance}
\al
\int_{B(O,r)}\frac{1}{d(y,O)^{\gamma}}dg(y)&=\int_{S^{n-1}}\int_{0}^{r}\frac{1}{s^{\gamma }}w(s,\theta)dsd\theta\\
&=\int_{S^{n-1}}\sum_{k=1}^{\infty}\int_{\frac{r}{2^{k}}}^{\frac{r}{2^{k-1}}}\frac{1}{s^{\gamma }}w(s,\theta)dsd\theta\\
&\leq \int_{S^{n-1}}\sum_{k=1}^{\infty}\frac{2^{k\gamma}}{r^{\gamma }}\int_{\frac{r}{2^{k}}}^{\frac{r}{2^{k-1}}}e^{C(\a)Ks^{1-\a}+\lam s^2}s^{n-1}ds d\theta\\
&\leq e^{C(\a)Kr^{1-\a}+\lam r^2}\int_{S^{n-1}}\sum_{k=1}^{\infty}\frac{2^{k\gamma }}{r^{\gamma }}\cdot\frac{r^{n}}{2^{n(k-1)}}d\theta\\
&\leq C(n,\gamma)e^{C(\a)Kr^{1-\a}+\lam r^2}r^{n-\gamma}.
\eal
\ee

Also, by \eqref{volume element} and the noncollapsing assumption \eqref{noncollapsing}, for $r\leq 1$ one has
\be\label{noncollapsing 2}
\al
\v(B(x,r))&=\int_{S^{n-1}}\int_0^r w(s,\theta)dsd\theta\\
&=r\int_{S^{n-1}}\int_0^1 w(rt,\theta)dtd\theta\\
&\geq r\int_{S^{n-1}}\int_0^1 \frac{1}{e^{C(\a)K+\lam}}r^{n-1}w(t,\theta)dtd\theta\\
&=\frac{\rho}{e^{C(\a)K+\lam}}r^n.
\eal
\ee
This and (\ref{case 1}) imply that
\be
\frac{1}{\v(B(x,r))} \int_{B(O,r)}\frac{1}{d(y,O)^{\a q}}dg(y)
\le
C(n,\a,q)e^{C(\a)Kr^{1-\a}+\lam r^2} \, \frac{e^{C(\a)K+\lam}}{\rho} \, r^{-\a q}.
\ee

\hspace{-.5cm}{\it Case 2:} If $d(x,O)> 2r$, then $d(y,O)\geq d(x,O)-d(x,y)\geq r$ for all $y\in B(x,r)$. Thus,
\be\label{case 2}
\int_{B(x,r)}\frac{1}{d(y,O)^{\a q}} dg(y)\leq r^{-\a q}\v(B(x,r)).
\ee

So in both cases we obtain
\[
||V||^*_{q,B(x,r)}\leq C(n,\lam,K,\a,\rho,q)Kr^{-\a}.
\]
This finishes the proof of the lemma. \qed \\

Now we are ready to prove part (b) of
Theorem \ref{thm volume element comparison}.

\hspace{-.5cm}{\it Proof of part (b).} Recall that $\psi=(\Delta s-\frac{n-1}{s})_+$. From \eqref{Laplace comparison}, we have
\be\label{bound of psi}
\psi\leq \frac{\lam}{3}s+\frac{C(\a)K}{s^{\a}}+|V|.
\ee
Notice that by \eqref{Lq bound of V}, we have
\be\label{L1 bound of V}
\oint_{B(x,r)}|V|dg\leq \frac{C(n,\lam, K,\a,\rho)K}{r^{\a}}.
\ee
Let $Q(r)=\frac{\v(B(x,r))}{r^n}$, then \eqref{differential of volume ratio}, \eqref{bound of psi} and \eqref{L1 bound of V} imply that
\[
\al
\frac{d}{dr}Q(r)&\leq \frac{1}{r^{n+1}}\int_{S^{n-1}}\int_0^r s\left(\frac{\lam}{3}s+\frac{C(\a)K}{s^{\a}}+|V|\right) w(s,\theta)dsd\theta\\
&\leq \left[\frac{\lam r}{3}+\frac{C(n,\lam,K,\a,\rho)K}{r^{\a}}\right]Q(r).
\eal
\]
The proof of the theorem is finished after integrating both sides of the above inequality. \\

Finally, the proof of part (c) follows directly by considering the special case where $V=\d L$ in \eqref{riccond}. \qed\\

In the following, we show that under a stronger assumption that $|V|$ is bounded, the results above can be improved. Especially, we can obtain the volume comparison theorem without assuming the volume noncollapsing condition \eqref{noncollapsing-0}.
\begin{corollary}
\lab{coro bounded V}

Suppose that \eqref{riccond} holds and moreover
\be\label{bounded V}
|V|\leq K
\ee Then the following conclusions are true.

(a). Let $s=d(y,x)$ be the distance from any point $y$ to some fixed point $x$,
and $\gamma:\ [0,s]\rightarrow\M$ a normal minimal geodesic with
 $\gamma(0)=x$ and $\gamma(s)=y$.
Then in the distribution sense,
\be\label{Laplace comparison bounded V}
\Delta s-\frac{n-1}{s}\leq \frac{\lam}{3}s+2 K.
\ee

(b). Let $w(s,\cdot)$ denote the volume element of the metric $g$ on $\M$
under geodesic polar coordinates. Then for any $0<s_1<s_2$, we have
\be\label{2volume element ratio}
\frac{w(s_2,\cdot)}{s_2^{n-1}}\leq e^{2 K s_2^{} +\lam s_2^2}\frac{w(s_1,\cdot)}{s_1^{n-1}}.
\ee
In particular, by letting $s_1\rightarrow 0$, we get
\be\label{2volume element}
w(s,\cdot)\leq e^{2 K s+\lam s^2}s^{n-1},\quad \forall s>0,
\ee
and hence
\be\label{2volume noninflation}
\vol(B(x,r))\leq e^{2 K r+\lam r^2}\frac{\vol(S^{n-1})}{n}r^n, \quad \forall r>0.
\ee

(c). For any $0<r_1<r_2$, we have
\be
\al
\frac{\vol(B(x,r_2))}{r_2^n}\leq e^{[\lam(r_2^2-r_1^2)+2K(r_2-r_1)]}\cdot\frac{\vol(B(x,r_1))}{r_1^n}.\\
\eal
\ee
\end{corollary}

\proof  When $|V|$ is bounded, Lemma \ref{lemma Lq bound of V} holds
without any other condition for $\a=0$ and $q=\infty$. The proof of the corollary is just special cases of
 those of Proposition \ref{prop laplacian comparison},
Theorem \ref{thm volume element comparison} with the parameter $\a =0$. \qed

\begin{remark}
Following the proofs above verbatim, one can see that similar results also hold if we relax assumption \eqref{riccond} to
\[
Ric+\frac{1}{2}\mathcal{L}_V g - \frac{1}{N-n}V\otimes V\geq -\lam g,
\]
and require the constant $\a$ in assumption \eqref{condition V} to be less than $1/2$. Here $N>n$ is any constant, and when $N=\infty$, we arrive at \eqref{riccond}.
\end{remark}

\medskip


\subsection{Bakry-\'Emery Ricci condition}

Next we consider the special case when $V$ is a gradient field, namely we are returning
to the Bakry-\'Emery Ricci curvature.
Let $V=\d L$ for some smooth function $L$ and assume
\be\label{gradient riccond}
\Ric+\Hess L\geq - \lam g
\ee
for some constant $\lam \geq0$.  Since one can perform integration by parts twice, the condition on the gradient of the potential can be relaxed to an integral one.

In this case, we will impose the following conditions on $L$:
\be\label{condition L}
|L(y)-L(z)|\leq K_1d(y,z)^\a,\  \textrm{and}\ \sup_{x\in\M, 0\leq r\leq 1}\left(r^{\beta}||\d L||_{q,\, B(x,r)}^*\right)\leq K_2
\ee
for any $y, z\in\M$ with $d(y,z)\leq 1$.  Here $K_1, K_2\geq0$, $0<\a<1$, $0\leq\beta<1$, and $q\geq 1$ are constants, and
\[
||\d L||_{q,\, B(x,r)}^*=\left(\oint_{B(x,r)}\vert \nabla L\vert^q(y) dV(y)\right)^{1/q}.
\]   Comparing with the more general case earlier,  we no longer need to assume $\M$ being non-collapsed.
Intuitively, the above assumptions on the potential function $L$ is that it is H\"older continuous and $|\nabla L|$ has singularity of order less $1$.  These include not only the popular special case when $|\nabla L|$ is bounded but also Ricci solitons with just bounded Ricci potentials. More detailed discussion will be given later in this section.  Since we are mainly concerned with local properties, the assumption on $L$ is
restricted to balls of size $1$ for convenience. The condition on $\nabla L$ is almost sharp by scaling considerations.

\begin{theorem}
\lab{thm.BE.vol}
Suppose that \eqref{gradient riccond} and \eqref{condition L} are satisfied.

(a) Let $s=d(x,y)$ be the distance from any point $y$ to some fixed point $x$, and $\gamma:\ [0,s]\rightarrow\M$ a normal minimal geodesic with $\gamma(0)=x$ and $\gamma(s)=y$. Then the following inequalities hold in the sense of distribution.
\be\label{Laplace comparison L}
\Delta s-\frac{n-1}{s}\leq \frac{\lam}{3}s+\frac{4K_1}{s^{1-\a}}+<\d L, \d s>, \quad \forall s<1,
\ee
moreover,
\be\label{differential of volume element ratio L}
\frac{\partial}{\partial s}{\frac{w(s,\theta)}{s^{n-1}}}\leq \left[\frac{\lam}{3}s+\frac{4K_1}{s^{1-\a}}+<\d L, \d s>\right]{\frac{w(s,\theta)}{s^{n-1}}}.
\ee Here $w=w(s, \theta)$ is the volume element in the geodesic polar coordinates which is regarded as $0$ on the cut locus.

(b) For any $0<r_1<r_2\leq 1$, we have
\be
\al
\frac{\vol(B(x,r_2))}{r_2^n}\leq e^{\lam(r_2^2-r_1^2)+K_2 (r_2-r_1)^{1-\beta}+4K_1 (r_2-r_1)^{\a}}\cdot\frac{\vol(B(x,r_1))}{r_1^n}.
\eal
\ee
\end{theorem}

\proof[Proof of (a)]

Again, we just prove the inequalities at smooth points of $s$.

First of all, from \eqref{Delta s}, we have
\be\label{Delta s L}
\Delta s\leq \frac{n-1}{s}-\frac{1}{s^2}\int_0^st^2\Ric(\frac{\pa}{\pa t}, \frac{\pa}{\pa t})ds.
\ee
According to \eqref{gradient riccond}, along the geodesic $\gamma(t)$, one has
\be\label{geodesic BERicci}
\Ric(\pa_t,\pa_t)\geq -\lam-\Hess L(\pa_t,\pa_t)=-\lam-\frac{d^2}{dt^2}L(\gamma(t)).
\ee

By using \eqref{geodesic BERicci} and integration by parts, we can rewrite \eqref{Delta s L} as
\be\label{Delta s and L}
\al
\Delta s- \frac{n-1}{s}&\leq\frac{1}{s^2}\int_0^s \left[t^2\frac{d^2}{dt^2}L(\gamma(t))+\lam t^2\right]dt\\
&\leq \frac{\lam}{3}s+\frac{1}{s^2}\left[t^2\frac{d}{dt}L(\gamma(t))\big|_0^s-2\int_0^st\frac{d}{dt}L(\gamma(t))dt\right]\\
&\leq \frac{\lam}{3}s+<\d L, \d s> -\frac{2}{s^2}\int_0^s t\frac{d}{dt}L(\gamma(t))dt\\
&=\frac{\lam}{3}s+<\d L, \d s>-\frac{2}{s}L(\gamma(s))+\frac{2}{s^2}\int_0^sL(\gamma(t))dt\\
&=\frac{\lam}{3}s+<\d L, \d s>-\frac{2}{s}[L(\gamma(s))-L(\gamma(0))]+\frac{2}{s^2}\int_0^s[L(\gamma(t))-L(\gamma(0))]dt\\
&\leq \frac{\lam}{3}s+<\d L, \d s>+\frac{4K_1}{s^{1-\a}}.
\eal
\ee This proves \eqref{Laplace comparison L}.  From this, \eqref{differential of volume element ratio L} follows by the formula $\partial_s w = \Delta s \, w$.\\

\proof[Proof of (b)]
From \eqref{Laplace comparison L}, we have
\[
\psi\leq \frac{\lam}{3}s+\frac{4K_1}{s^{1-\a}}+|\d L|.
\]
Notice that by \eqref{condition L}, we have
\be\label{L1 gradient L}
\oint_{B(x,r)}|\d L|dV\leq \left(\oint_{B(x,r)}|\d L|^q dV\right)^{1/q}\leq \frac{K_2}{r^{\beta}}.
\ee
Let $Q(r)=\frac{\v(B(x,r))}{r^n}$, then \eqref{L1 gradient L} and Proposition \ref{proposition differential of volume ratio} imply that
\[
\al
\frac{d}{dr}Q(r)&\leq \frac{1}{r^{n+1}}\int_{\mb{S}^{n-1}}\int_0^r s\left(\frac{\lam}{3}s+\frac{4K_1}{s^{1-\a}}+|\d L|\right) w(s,\theta)dsd\theta\\
&\leq (\frac{\lam r}{3}+\frac{4K_1}{r^{1-\a}}+\frac{K_2}{r^{\beta}})Q(r).
\eal
\]
This finishes the proof of part (b) after integrating both sides of the above inequality. \qed \\

It is worth mentioning  that, when $V=\d L$ in \eqref{riccond}, the assumption \eqref{condition L} is actually a weaker condition than \eqref{condition V=dL} and \eqref{noncollapsing-0} due to the following lemma.

\begin{lemma}
Suppose that
\[
Ric+\Hess L\geq -\lam g,
\] and assume that \eqref{condition V=dL} and \eqref{noncollapsing-0} hold, i.e.,
\[
|\nabla L(x)| \le \frac{K}{d(x, O)^\a},
\quad
\text{and} \quad
\vol(B(x,1))\geq \rho.
\]
Then for any $q\in(n,\frac{n}{\a})$, we have
\be\label{special case 1}
|L(y)-L(z)|\leq \frac{2K}{1-\a}d(y,z)^{1-\a},
\ee
for any $y,z\in\M$, and
\be\label{special case 2}
\sup_{x\in\M,0< r\leq 1}r^{\a}||\d L||^*_{q,B(x,r)}\leq C(n,K,\lam,\a,\rho)K.
\ee
\end{lemma}

\begin{proof}
Since the proof of \eqref{special case 2} is given in Lemma \ref{lemma Lq bound of V}, we only need to show \eqref{special case 1}. Suppose that $\gamma:\ [0,d(y,z)]\rightarrow\M$ is a minimal geodesic from $y$ to $z$. We separate the proof into two cases.\\

\hspace{-.5cm}\textit{Case 1:} If $d(y,O)\geq d(y,z)$, then for any $t\in[0,d(y,z)]$, $d(\gamma(t),O)\geq d(y,O)-d(y,\gamma(t))=d(y,O)-t$. Thus,
\[
\al
|L(y)-L(z)|&=\big|\int_0^{d(y,z)}<\d L, \gamma'(t)>dt\big|\\
&\leq K\int_0^{d(y,z)}\frac{1}{[d(y,O)-t]^\a} dt\\
&=\frac{K}{1-\a}\left(d(y,O)^{1-\a}-[d(y,O)-d(y,z)]^{1-\a}\right)\\
&\leq \frac{K}{1-\a}d(y,z)^{1-\a}.
\eal
\]

\hspace{-.5cm}\textit{Case 2:} If $d(y,O)< d(y,z)$, then
\[
\al
|L(y)-L(z)|&\leq K\int_0^{d(y,z)}\frac{1}{d(\gamma(t),O)^\a}dt\\
&=K\int_0^{d(y,O)}\frac{1}{d(\gamma(t),O)^\a}dt+K\int_{d(y,O)}^{d(y,z)}\frac{1}{d(\gamma(t),O)^\a}dt\\
&\leq K\int_0^{d(y,O)}\frac{1}{[d(y,O)-t]^\a} dt+K\int_{d(y,O)}^{d(y,z)}\frac{1}{[t-d(y,O)]^\a} dt\\
&=\frac{K}{1-\a}\left(d(y,O)^{1-\a}+[d(y,z)-d(y,O)]^{1-\a}\right)\\
&\leq \frac{2K}{1-\a}d(y,z)^{1-\a}.
\eal
\]
This proves \eqref{special case 1}, and hence finishes the proof of the lemma.
\end{proof}


The next corollary treats the case of gradient Ricci solitons for which the results are the strongest.
A complete Riemannian manifold $(M, g_{ij})$ is a so called gradient Ricci soliton if the following equation is satisfied
\be\label{Ricci soliton}
R_{ij}+\d_i\d_j L=\lambda g_{ij},
\ee
where $\lam$ is a constant, and $L$ is a smooth function on $\M$. The cases where $\lam>0$, $=0$, and $<0$ correspond to shrinking, steady and expanding solitons, respectively.

Ricci solitons not only represent self-similar solutions of the Ricci flow, but more importantly, they appear as singularity models, i.e., the dilation limits of singular solutions,  in many circumstances in the Ricci flow. Thus, the study of Ricci solitons is critical for the singularity analysis in the Ricci flow. We refer readers to \cite{Cao1}, \cite{Cao2} and references therein for more information on the Ricci solitons.

For a Ricci soliton \eqref{Ricci soliton}, it is well known that
\be\label{unnormalized soliton}
R+|\d L|^2-2\lam L=C_0,
\ee
where $R$ is the scalar curvature of $\M$, and $C_0$ is a constant. Therefore, for shrinking and expanding solitons, i.e., when $\lam\neq 0$, we may add $L$ by constant so that it satisfies the normalization condition
\be\label{normalized shriking and expanding solitons}
R+|\d L|^2-2\lam L=0.
\ee
For steady soliton, the above equation becomes
\be\label{unnormalized steady soliton}
R+|\d L|^2=C_0.
\ee
If $C_0=0$, then nothing needs to be done. Actually, Z.-H. Zhang \cite{zZh} and B.-L. Chen \cite{Chen} showed that $R\geq 0$ for steady solitons. Thus, when $C_0=0$, one immediately gets from \eqref{unnormalized steady soliton} that $R\equiv0$. Then by the following evolution equation of $R$,
\[
\Delta R=<\d R, \d L>+2|Ric|^2,
\]
we obtain $Ric\equiv0$.

Therefore, for steady solitons, let us just consider the case where $C_0\neq 0$ in \eqref{unnormalized steady soliton}. However, in this case, it is impossible to normalize the soliton so that $C_0=0$. On the other hand, we may rescale the metric $g$ by a constant, so that the steady soliton can be normalized as
\be\label{normalized steady soliton}
R+|\d L|^2=1.
\ee

Assume that $(M, g)$ is a gradient Ricci soliton with normalization \eqref{normalized shriking and expanding solitons} if it is a shrinking or expanding soliton, or \eqref{normalized steady soliton} if it is a steady soliton. Again, by the result in \cite{zZh} and \cite{Chen}, we have $R\geq0$ when $\lam\geq 0$, and $R\geq -C_1(n)$ when $\lam<0$. Hence, if we further assume $|L|\leq K$, then it follows from \eqref{normalized shriking and expanding solitons} and \eqref{normalized steady soliton} that
\be\label{dL and L}
|\d L|\leq \Lambda(n, \lam, K),
\ee
where
\be\label{Lambda}
\Lambda(n,\lam,K)=\left\{\begin{array}{lll} \sqrt{2\lam K},\ &if\ \lam>0;\\ 1,\ &if\ \lam=0;\\ \sqrt{-2\lam K+C_1(n)},\ &if\ \lam<0. \end{array}\right.
\ee
This shows that the assumptions in Corollary \ref{coro bounded V} are satisfied. Therefore, based on the previous argument and Corollary \ref{coro bounded V}, we have

\begin{corollary}  (Gradient Ricci soliton case)\label{coro Ricci soliton}
Suppose that $(M, g_{ij})$ is a gradient Ricci soliton \eqref{Ricci soliton} satisfying \eqref{normalized shriking and expanding solitons} when $\lam\neq0$ or \eqref{normalized steady soliton} when $\lam=0$. Assume further that
\be\label{bounded L}
|L|\leq K,\ in\ B(O, 2\delta)
\ee
for some fixed point $O\in\M$ and radius $\delta$. Let $\Lambda(n,\lam,K)$ be the constant in \eqref{Lambda}. Then the following statements hold.

(a). Let $s=d(y,x)$ be the distance between any two points $x,y\in B(O,\delta)$,
and $\gamma:\ [0,s]\rightarrow\M$ a normal minimal geodesic with
 $\gamma(0)=x$ and $\gamma(s)=y$.
Then in the distribution sense,
\be\label{Laplace comparison bounded V}
\Delta s-\frac{n-1}{s}\leq \frac{|\lam|}{3}s+2\Lambda(n,\lam,K).
\ee

(b). Let $s$ be defined as in (a), and $w(s,\cdot)$ denote the volume element of the metric $g$
under geodesic polar coordinates. Then for any $0<s_1<s_2<d(x,\pa B(O,\delta))$,
we have
\be\label{2volume element ratio}
\frac{w(s_2,\cdot)}{s_2^{n-1}}\leq e^{2\Lambda(n,\lam,K) s_2+|\lam| s_2^2}\frac{w(s_1,\cdot)}{s_1^{n-1}}.
\ee
In particular, by letting $s_1\rightarrow 0$, we get
\be\label{2volume element}
w(s,\cdot)\leq e^{2\Lambda(n,\lam,K) s+|\lam| s^2}s^{n-1},
\ee
for any $0<s<d(x,\pa B(O,\delta))$, and hence
\be\label{2volume noninflation}
\vol(B(x,r))\leq C(n)e^{2\Lambda(n,\lam,K) r+|\lam| r^2}r^n,
\ee
for any $0<r<d(x,\pa B(O,\delta))$.

(c). For any $x\in B(O,\delta)$ and $0<r_1<r_2<d(x,\pa B(O,\delta))$, we have
\be
\al
\frac{\vol(B(x,r_2))}{r_2^n}\leq e^{[|\lam|(r_2^2-r_1^2)+2\Lambda(n,\lam,K)(r_2-r_1)]}\cdot\frac{\vol(B(x,r_1))}{r_1^n}.\\
\eal
\ee
\end{corollary}

\begin{remark}
Note that this is a local volume ratio comparison result for gradient Ricci solitons.  A volume
upper bound in large scale for gradient shrinking solitons is proven in H.D. Cao and D.T. Zhou \cite{CZ}.
\end{remark}

\begin{remark}
For gradient shrinking Ricci solitons, Z.L. Zhang \cite{zlZh} showed that if the volume of $M$ is bounded below and the diameter is bounded above, then the potential function is bounded, and after a conformal transformation, the Ricci curvature also becomes bounded.
\end{remark}
\medskip

\section{Local isoperimetric and sobolev constants estimates}

In this section, we use the volume comparison result in the previous section and follow
the technique and arguments in X.Z. Dai-G.F. Wei-Z.L. Zhang \cite{DWZ} to prove isoperimetric inequality and Sobolev inequality
on manifolds satisfying the modified Ricci bound.
The basic assumptions are \eqref{riccond}, \eqref{condition V} and the volume noncollapsing condition \eqref{noncollapsing-0}.
However  if assumption  \eqref{riccond}, \eqref{condition V}  hold with $\a=0$, namely the vector field $V$ is bounded, then the
noncollapsing condition is not necessary.

\begin{proposition}\label{proposition half volume}
Assume that the basic assumptions \eqref{riccond}, \eqref{condition V} and the noncollapsing condition \eqref{noncollapsing-0} hold. Or just assume  that \eqref{riccond}, \eqref{condition V} hold for $\a=0$.
There exists $r_0=r_0(n,\lam,K,\a,\rho)$ and $\delta=\delta(n)$, such that for any $r\leq r_0$,
\be\label{half volume}
\frac{\vol(B(x,\delta r))}{\vol(B(x,r))}\leq \frac{1}{2}.
\ee
\end{proposition}

\begin{proof}
First of all, if follows from Theorem \ref{thm volume element comparison} (b) that, for  $r_0=r_0(n,\lam,K,\a,\rho)<1$ such that
\[
e^{C(n,\lam,K,\a,\rho)(\lam r_0^2+Kr_0^{1-\a})}\leq \frac{3}{2},
\]
one has
\be\label{volume doubling}
\frac{\v(B(x,r_1))}{\v(B(x,r_2))}\geq \frac{2}{3}\frac{r_1^n}{r_2^n},
\ee
for any $x\in \M$ and $0<r_1<r_2\leq r_0$.

For any $r\leq r_0$, let $\delta<\frac{1}{3}$, and choose $x'\in B(x,r)$ with $d=d(x,x')=\frac{1-\delta}{2}r\geq \frac{1}{3}r$. Then
\[
B(x,\delta r)\subset B(x',d+\delta r)\setminus B(x', d-\delta r)\subset B(x',d+\delta r)\subset B(x,r).
\]
Therefore,
\[
\frac{\v(B(x,\delta r))}{\v(B(x,r))}\leq \frac{\v(B(x',d+\delta r))-\v(B(x',d-\delta r))}{\v(B(x',d+\delta r))}=1-\frac{\v(B(x',d-\delta r))}{\v(B(x',d+\delta r))}.
\]
By \eqref{volume doubling}, we have
\[
\frac{\v(B(x',d-\delta r))}{\v(B(x',d+\delta r))}\geq\frac{2}{3}\left(\frac{d-\delta r}{d+\delta r}\right)^n=\frac{2}{3}\left(\frac{1-3\delta}{1+\delta}\right)^n.
\]
If we choose $\delta=\delta(n)$ small so that
\[
\left(\frac{1-3\delta}{1+\delta}\right)^n\geq\frac{3}{4},
\]
then
\[
\frac{\v(B(x',d-\delta r))}{\v(B(x',d+\delta r))}\geq \frac{1}{2},
\]
and hence
\[
\frac{\v(B(x,\delta r))}{\v(B(x,r))}\leq \frac{1}{2}.
\]
\end{proof}

\begin{remark}
If one assumes the volume noncollapsing condition, then we can also choose $\delta=\delta(n,\lam,K,\a,\rho)$ small, such that the same conclusion holds for any $r\leq 1$.
\end{remark}

\begin{lemma}[Gromov \cite{Gro}]\label{lem gromov} Let $\M^n$ be a complete Riemannian manifold and $H$ any hypersurface dividing $\M$ into two parts, $\M_1$ and $\M_2$. For any Borel subsets $W_i\subset \M_i$, there exists $x_1$ in one of $W_i$, say $W_1$, and a subset $W$ in the other part $W_2$, such that
\[
\vol(W)\geq \frac{1}{2}\vol(W_2),
\]
and for any $x_2\in W$, there is a unique minimal geodesic between $x_1$ and $x_2$ which intersects $H$ at some $z$ with
\[
d(x_1,z)\geq d(x_2,z).
\]
\end{lemma}

\begin{lemma}\label{lem H}
Let $H$, $W$ and $x_1$ be as in the above lemma. Assume that $D=\sup_{x\in W}d(x_1,x)\leq 1$, then
\be\label{volume W H}
\vol(W)\leq 2^{n-1}D\vol(H')+\left[2^{n-2}\lam D^2+C(n,\lam,K,\a,\rho)D^{1-\a}\right]\vol(B(x,D)),
\ee
where $H'$ is the set of intersection points with $H$ of geodesic $\gamma_{x_1,x}$ for all $x\in W$.
\end{lemma}

\begin{proof}
Let $\Gamma\subset S_x$ be the set of unit vectors at $x$ such that $\gamma_{v}=\gamma_{x_1,x_2}$ for some $x_2\in W$.
Since
\[
\frac{\pa}{\pa s}\frac{w(s,\theta)}{s^{n-1}}=\frac{s\pa_s w-(n-1)w}{s^n}\leq \frac{\psi w(s,\theta)}{s^{n-1}},
\] after integration we have
\[
w(r,\theta)\leq 2^{n-1}w(t,\theta)+2^{n-1}\int_{t}^{r}\psi(s,\theta)w(s,\theta)ds,
\]
whenever $\frac{r}{2}\leq t\leq r$.

For any $\theta \in \Gamma$, denote by $r_1(\theta)$ and $r_2(\theta)$ the minimum and maximum radius such that $exp_{x_1}r\theta\in W$, respectively. Let $r(\theta)$ be the radius such that $exp_{x_1}(r\theta)\in H$. Then from Lemma \ref{lem gromov}, we have $r(\theta)\geq \frac{1}{2}r_2(\theta)$, and hence
\[
\al
\v(W)&\leq \int_{\Gamma}\int_{r_1(\theta)}^{r_2(\theta)}w(t,\theta)dtd\theta\\
&\leq 2^{n-1}\int_{\Gamma}\int_{r_1(\theta)}^{r_2(\theta)}w(r(\theta),\theta)dtd\theta+2^{n-1}\int_{\Gamma}\int_{r_1(\theta)}^{r_2(\theta)}\int_{r(\theta)}^{t}\psi(s,\theta)w(s,\theta)dsdtd\theta\\
&\leq 2^{n-1}D\int_{\Gamma}w(r(\theta),\theta)d\theta+2^{n-1}r_2(\theta)
\int_{\Gamma}\int_{r(\theta)}^{r_2(\theta)}[\frac{\lam}{3}s+\frac{C(\a)K}{s^{\a}}+|V|(\gamma(s))]w(s,\theta)dsd\theta\\
&\leq 2^{n-1}D\v(H')+\left[\frac{2^{n-1}}{3}\lam r_2^2(\theta)+C(n,\a)Kr_2^{1-\a}(\theta)\right]\int_{\Gamma}\int_{0}^{D}w(s,\theta)dsd\theta\\
&\quad + 2^{n-1}D\int_{B(x,D)}|V|dV\\
&\leq 2^{n-1}D\v(H')+\left[2^{n-2}\lam D^2+C(n,\lam,K,\a,\rho)KD^{1-\a}\right]\v(B(x,D)).
\eal
\]
In the last step above, we have used \eqref{L1 bound of V}.  We have also used Proposition \ref{prop laplacian comparison} when
going from the 2nd to the 3rd line.
\end{proof}

From Lemmas \ref{lem gromov} and \ref{lem H}, we immediately have
\begin{corollary}
Let $H$ be any hypersurface dividing $\M$ into two parts $\M_1$ and $\M_2$. For any ball $B=B(x,r)$, $r\leq 1/2$, we have
\be
\al
&\min\left(\vol(B\cap\M_1), \vol(B\cap\M_2)\right)\\
\leq &2^{n+1}r\vol(H')+\left[2^{n}\lam r^2+C(n,\lam,K,\a,\rho)Kr^{1-\a}\right]\vol(B(x,2r)).
\eal
\ee
\end{corollary}

\begin{corollary}\label{cor volume ball surface}
Assume that \eqref{riccond}, \eqref{condition V} and \eqref{noncollapsing-0} hold.
Or just assume that \eqref{riccond}, \eqref{condition V} hold for $\a=0$.
There is a $r_0=r_0(n,\lam,K,\a,\rho)$ such that for any $r\leq r_0$, and geodesic ball $B(x,r)$ which is divided equally by $H$, we have
\be\label{volume ball surface}
\vol(B(x,r))\leq 2^{n+3}r\vol(H\cap B(x,2r)).
\ee
\end{corollary}
\begin{proof}
By \eqref{volume comparison}, we have
\[
\vol(B(x,2r))\leq 2^ne^{C(n,\lam,K,\a,\rho)(\lam r^2+Kr^{1-\a})}\vol(B(x,r)).
\]
Thus, to get the result in the corollary, one only needs to choose $r_0\leq\frac{1}{2}$ so that
\[
2^ne^{C(n,\lam,K,\a,\rho)(\lam r_0^2+Kr_0^{1-\a})}\left[2^{n}\lam r_0^2+C(n,\lam,K,\a,\rho)Kr_0^{1-\a}\right]\leq 1/4.
\]
\end{proof}

From Proposition \ref{proposition half volume}, Corollary \ref{cor volume ball surface} and the results in section 2, we can prove the following isoperimetric inequality.

\begin{theorem}\label{isoperimetric constant}
Assume that \eqref{riccond}, \eqref{condition V} and \eqref{noncollapsing-0} hold.
Or just assume that \eqref{riccond}, \eqref{condition V} hold for $\a=0$.
There exists an $r_0=r_0(n,\lam,K,\a,\rho)$ such that for any $r\leq r_0$,
\[
ID_n^*(B(x,r))\leq C(n)r.
\]
Here $ID_n^*(B(x,r))$ is the isoperimetric constant defined by
\[
ID_n^*(B(x,r))=\vol(B(x,r))^{\frac{1}{n}}\cdot \sup_{\Omega}\left\{\frac{\vol(\Omega)^{\frac{n-1}{n}}}{\vol(\partial \Omega)}\right\},
\]
where the supremum is taken over all smooth domains $\Omega\subset B(x,r)$ with $\partial \Omega\cap \partial B(x,r)=\emptyset$.
\end{theorem}

\begin{proof}
By Proposition \ref{proposition half volume}, there exists an $r_1$ and $\delta=\delta(n)$ so that for any $r\leq r_1$,
\[
\frac{\v(B(x,2\delta r))}{\v(B(x,\frac{1}{10}r))}\leq \frac{1}{2},\quad \forall x\in\M.
\]

We will prove the theorem for any $r\leq r_0=\delta r_1$. Given any $x\in \M$, let $\Omega$ be a smooth connected subdomain of $B(x, r)$ whose boundary $\partial\Omega$ divides $\M$ into two parts $\Omega$ and $\Omega^c$. For any $y\in \Omega$, let $r_y$ be the smallest radius such that
\[
\v(B(y,r_y)\cap\Omega)=\v(B(y,r_y)\cap\Omega^c)=\frac{1}{2}\v(B(y,r_y)).
\]
Since $\Omega\subset B(y,2r)$,
\[
\al
\v(B(y,\frac{1}{10}r))&\geq 2\v(B(y,2 r))\\
&\geq 2\v(B(y,r_y)\cap B(y,2r))\geq 2\v(B(y,r_y)\cap \Omega)=\v(B(y,r_y)).
\eal
\]
Thus, $r_y\leq \frac{1}{10}r$.

By \eqref{volume ball surface},
\[
\vol(B(y,r_y))\leq 2^{n+3}r_y\vol(H\cap B(y,2r_y)).
\]
Since
$$\Omega\subset\bigcup_{y\in\Omega}B(y,2r_y),$$
by Vitali Covering Lemma, we may choose a countable family of disjoint balls $B_i=B(y_i, 2r_{y_i})$ such that $\Omega\subset\cup_iB(y_i, 10r_{y_i})$. By \eqref{volume doubling}, we have
\[
\frac{\v(B(y, 2r_y))}{\v(B(y,10r_y))}\geq \frac{2}{3\cdot 5^n}.
\]
Hence,
\[
\sum_i\v(B_i)\geq\frac{2}{3\cdot 5^n}\sum_{i}\v(B(y_i,10r_{y_i}))\geq\frac{2}{3\cdot 5^n}\v(\Omega),
\]
and
\[
\sum_i\v(B(y_i,r_{y_i}))\geq \frac{4}{9\cdot 10^n}\v(\Omega).
\]

Since $B_i$'s are disjoint, from \eqref{volume ball surface}, we get
\[
\v(\partial\Omega)\geq \sum_i\v(B_i\cap \partial \Omega)\geq 2^{-(n+3)}\sum_ir^{-1}_{y_i}\v(B(y_i,r_{y_i})).
\]
Therefore,
\[
\al
\frac{\v(\Omega)^{\frac{n-1}{n}}}{\v(\partial \Omega)}&\leq \left(\frac{9\cdot 10^n}{4}\right)^{\frac{n-1}{n}}2^{n+3}\frac{\left(\sum_i\v(B(y_i,r_{y_i}))\right)^{\frac{n-1}{n}}}{\sum_ir_{y_i}^{-1}\v(B(y_i,r_{y_i}))}\\
&\leq 10^{2n+5}\frac{\sum_i\v\left(B(y_i,r_{y_i})\right)^{\frac{n-1}{n}}}{\sum_ir_{y_i}^{-1}\v(B(y_i,r_{y_i}))}\\
&\leq 10^{2n+5}\sup_i\frac{\v(B(y_i,r_{y_i}))^{\frac{n-1}{n}}}{r_{y_i}^{-1}\v(B(y_i,r_{y_i}))}\\
&=10^{2n+5}\sup_i\left(\frac{r_{y_i}^n}{\v(B(y_i,r_{y_i}))}\right)^{1/n}.
\eal
\]
On the other hand, by \eqref{volume doubling},
\[
\al
\v(B(y_i, r_{y_i}))\geq \frac{2}{3}\cdot\frac{r_{y_i}^n}{r_1^n}\v(B(y_i,r_1))\geq \frac{4}{3}\cdot\frac{r_{y_i}^n}{r_1^n}\v(B(y_i,2r))\geq \frac{4\delta^n}{3}\cdot\frac{r_{y_i}^n}{r_0^n}\v(B(x,r)).
\eal
\]
Therefore,
\[
\frac{\v(\Omega)^{\frac{n-1}{n}}}{\v(\partial \Omega)}\leq 10^{2n+5}\left(\frac{3r^n}{4\delta^n\v(B(x,r))}\right)^{1/n}.
\]
This finishes the proof of the theorem.
\end{proof}

It is well known that the isoperimetric inequality above is equivalent to the following Sobolev inequality.

\begin{corollary}\label{sobolev}
Under the same assumptions as in the above theorem, we have the following Sobolev inequalities.
\be\label{L1 Sobolev}
\left(\oint_{B(x,r)}|f|^{\frac{n}{n-1}}dg\right)^{\frac{n-1}{n}}\leq C(n)r\oint_{B(x,r)}|\d f|dg,
\ee
and
\be\label{L2 Sobolev}
\left(\oint_{B(x,r)}|f|^{\frac{2n}{n-2}}dg\right)^{\frac{n-2}{n}}\leq C(n)r^2\oint_{B(x,r)}|\d f|^2dg,
\ee
for any $f\in C_0^{\infty}(B(x,r))$ and $r\leq r_0$.\\
\end{corollary}

By setting $V=\d L$ for some smooth function $L$ in \eqref{riccond}, it immediately follows from Theorem \ref{isoperimetric constant} and Corollary \ref{sobolev} that

\begin{corollary}
Suppose that
\[
Ric+\Hess L\geq -\lam g,
\]
and
\[
|\nabla L(x)| \le \frac{K}{d(x, O)^\a}
\]
for all $x \in \M$, a fixed point $O \in \M$,
 constants $K\geq0$, $\a \in [0, 1)$, and assume in addition that
 the noncollapsing condition (\ref{noncollapsing-0}) holds if $\a>0$. Then there exists an $r_0=r_0(n,\lam,K,\a,\rho)$ such that for any $r\leq r_0$,\\
(a) $ID_n^*(B(x,r))\leq C(n)r,$ and\\
(b) \be
\left(\oint_{B(x,r)}|f|^{\frac{n}{n-1}}dg\right)^{\frac{n-1}{n}}\leq C(n)r\oint_{B(x,r)}|\d f|dg,
\ee
for any $f\in C_0^{\infty}(B(x,r))$.
\end{corollary}

On the other hand, by a similar argument, we can also show the following

\begin{corollary}\label{isoperimetric constant L}
Suppose that
\[
Ric+\Hess L\geq -\lam g,
\]
with $L$ satisfying \eqref{condition L}, i.e.,
\[
|L(y)-L(z)|\leq K_1d(y, z)^\a,\  \textrm{and}\ \sup_{x\in\M, 0\leq r\leq 1}\left(r^{\beta}||\d L||_{q,\, B(x,r)}^*\right)\leq K_2.
\]
Then there exists an $r_0=r_0(n,\lam,K_1,K_2,\a,\beta)$ such that for any $r\leq r_0$,\\
(a) $ID_n^*(B(x,r))\leq C(n)r,$ and\\
(b) \be
\left(\oint_{B(x,r)}|f|^{\frac{n}{n-1}}dg\right)^{\frac{n-1}{n}}\leq C(n)r\oint_{B(x,r)}|\d f|dg,
\ee
for any $f\in C_0^{\infty}(B(x,r))$.
\end{corollary}  Note that the non-collapsing condition is not needed in this corollary.
In all the results in this section, if $\alpha=0$ in \eqref{condition V}, i.e., $|V|$ is bounded,  then the non-collapsing condition is not
needed.
\medskip

\section{Gradient estimates, mean value inequality, heat kernel bounds and  cut-off functions}

In this section, we first state local gradient estimates and mean value inequalities for solutions of the
 Poisson equation and the heat equation. The general method of proof is more or less standard,
using the Sobolev inequality in the previous section and Moser's iteration.
However, comparing with the standard case, we need to deal with the Ricci
curvature term in a different manner. For this reason, a proof in the Poisson equation case is presented
in detail in Appendix I while the case for the heat equation is omitted.

These results allow us to establish the Gaussian upper and lower bounds of the heat kernel. Finally, we use the heat kernel to construct a good cut-off function which will be needed in the proofs of the main results.

In the following context, we denote by
\[
||f||^*_{q, B(x,r)}=\left(\oint_{B(x,r)}|f|^q\right)^{1/q}.
\]
\begin{theorem}[Elliptic gradient estimate]
\lab{thmpoisgrad}
Assume that \eqref{riccond}, \eqref{condition V} and \eqref{noncollapsing-0} hold, i.e.,
\[
\Ric+\frac{1}{2}\mc{L}_{V} g\geq - \lam g; \quad
|V|(y)\leq \frac{K}{d(y,O)^\a},\ \a \in [0, 1);
\quad \vol (B(x, 1)) \ge \rho.
\]
Then there exists a positive constant  $r_0=r_0(n,\lam,K,\a,\rho)$ such that,  for any $x\in\M$, $0<r\leq r_0$ and smooth functions $u$ and $f$ satisfying the equation
\[
\Delta u=f,\ in\  B(x,r),
\]
we have
\[
\sup_{B(x,\frac{1}{2}r)}|\d u|^2\leq C(n,\lam,K,\a,\rho)r^{-2}\left[(||u||^*_{2,B(x,r)})^2+(||f||^*_{2q,B(x,r)})^2\right],
\]
for any $q>n/2$.  Moreover
\[
\sup_{B(x,\frac{1}{2}r)} u^2\leq C(n,\lam,K,\a,\rho) \left[(||u||^*_{2,B(x,r)})^2+(||f||^*_{q ,B(x,r)})^2\right].
\]
In particular if $\a=0$, then the conclusions hold without the non-collapsing condition.
\end{theorem}
\begin{proof}
See Appendix I.
\end{proof}

By a similar argument, we get
\begin{corollary}
Suppose that
\[
Ric+\Hess L\geq -\lam g,
\]
with $L$ satisfying \eqref{condition L},i.e.,
\[
|L(y)-L(z)|\leq K_1d(y, z)^\a,\  \textrm{and}\ \sup_{x\in\M, 0\leq r\leq 1}\left(r^{\beta}||\d L||_{q,\, B(x,r)}^*\right)\leq K_2.
\]
Then there exists a constant $r_0=r_0(n,\lam,K_1,K_2, \a,\beta)$, such that for any $x\in\M$, $0<r\leq r_0$ and smooth functions $u$ and $f$ satisfying the equation
\[
\Delta u=f
\]
in $B(x,r)$, we have
\[
\sup_{B(x,\frac{1}{2}r)}|\d u|^2\leq C(n,\lam,K_1,K_2,\a,\beta)r^{-2}\left[(||u||^*_{2,B(x,r)})^2+(||f||^*_{2q,B(x,r)})^2\right],
\]
for any $q>n/2$. Moreover,
\[
\sup_{B(x,\frac{1}{2}r)} u^2\leq C(n,\lam,K_1,K_2,\a,\beta) \left[(||u||^*_{2,B(x,r)})^2+(||f||^*_{q ,B(x,r)})^2\right].
\]
\end{corollary}

As mentioned at the beginning of the section, the same argument also implies the corresponding result for the heat
equation:
\begin{theorem}[Parabolic gradient estimate]
\lab{thmheatgrad}
Suppose the same conditions as in Theorem \ref{thmpoisgrad} hold.
Then there exists a positive constant $r_0=r_0(n,\lam,K,\a,\rho)$ such that,  for any $x\in\M$, $0<r\leq r_0$ and smooth functions $u$ and $f$ satisfying the heat equation
\[
\Delta u(x, t) - \partial_t u(x, t)=f(x)
\]
in $Q(x, t, r)=B(x,r) \times [t-r^2, t]$, we have
\[
\sup_{Q(x, t, \frac{1}{2}r)}|\d u|^2\leq C(n,\lam,K,\a,\rho)r^{-2}\left[(||u||^*_{2, Q(x, t, r)})^2+
(||f||^*_{2q, B(x, r)})^2\right],
\]
for any $q>\frac{n}{2}$.  Moreover,
\[
\sup_{Q(x, t, \frac{1}{2}r)} u^2\leq C(n,\lam,K,\a,\rho) \left[(||u||^*_{2, Q(x, t, r)})^2+
(||f||^*_{q, B(x, r)})^2\right].
\]
\end{theorem}

Since the Sobolev inequality,
 volume comparison, volume non-collapsing,  and gradient bound holds, it is well known that the
following heat kernel bounds hold.  The proof is omitted.

\begin{theorem}[Heat kernel bounds]
\lab{thmheatkernel}
Under the same conditions as in Theorem \ref{thmpoisgrad},
for $G=G(x, t; y, 0)$, the heat kernel on the manifold $\M$,
the following bounds hold for some positive constants $C_i=C_i(n,\lam,K,\a,\rho)$, $i=1,\cdots,4$.
\[
C_1 t^{- \frac{n}{2}}e^{\frac{-C_2d^2(x,y)}{t}}\leq G(x, t;y, 0)\leq
C_3t^{-\frac{n}{2}}e^{-\frac{d^2(x,y)}{C_4t}},\ \forall x,y\in \M,\ \textrm{and}\  0<t\leq 1.
\]
\[
|\nabla_x G(x, t;y, 0)|\leq
C_3t^{-\frac{n+1}{2}}e^{-\frac{d^2(x,y)}{C_4t}},\ \forall x,y\in \M,\ \textrm{and}\  0<t\leq 1.
\]
\[
|\partial_t G(x, t;y, 0)|\leq
C_3t^{-\frac{n+2}{2}}e^{-\frac{d^2(x,y)}{C_4t}},\ \forall x,y\in \M,\ \textrm{and}\  0<t\leq 1.
\]
\end{theorem}

Similar to Corollary 3.2 in \cite{PeWe2}, we may also use the Sobolev inequality in Corollary \ref{sobolev} to prove the following maximum principle. Since the proof is identical, it is skipped.

\begin{theorem}[Maximum principle]\label{thm maximum principle}
Under the same conditions as in Theorem \ref{thmpoisgrad}, there is an $r_0=r_0(n,\lam,K,\a,\rho)$ so that for any $r\leq r_0$ and function $u$ satisfying
\[
\Delta u\geq f,\ in\ B(x,r),
\]
we have
\[
\sup_{B(x,r)}u\leq \sup_{\partial B(x,r)} u+C(n)r^{2}||f||^*_{q,B(x,r)},
\]
where $q>\frac{n}{2}$.
\end{theorem}

To prove the volume convergence and cone rigidity for Gromov-Hausdorff limits, another important tool is to construct an appropriate cut-off function (see e.g., Theorem 6.33 in \cite{ChCo1}). Since we have obtained Sobolev inequality \eqref{L2 Sobolev} and gradient estimate (Theorem \ref{thmpoisgrad}), we may wish to follow the proof of Theorem 6.33 in \cite{ChCo1} to show the existence of the cut-off function in the following lemma. However, the proof may be quite technical.  Instead, we adapt a method in \cite{BZ} and an covering argument in \cite{CN}  (see also \cite{Bam}) to construct the cut-off function through the heat kernel estimates in Theorem \ref{thmheatkernel}. This method may have a broader application.

\begin{lemma}[Cut-off function]\label{lem cutoff}
Suppose again that  the basic assumption \eqref{riccond}, \eqref{condition V} and the volume non-collapsing condition \eqref{noncollapsing-0} hold. Then for any $x_0\in \M$ and $0<r_0\leq 5$, there exists a function $\phi\in C^2(\M)$ satisfying\\
(1) $0\leq \phi\leq1$ everywhere.\\
(2) $supp\ \phi \subset\subset B(x_0,1.9r_0)$.\\
(3) $\phi=1$ in $B(x_0,1.1r_0)$.\\
(4) $|\d \phi|<C(n,\lam,K,\a,\rho)r_0^{-1}$ and $|\Delta \phi|<C(n,\lam,K,\a,\rho)r_0^{-2}$.
\end{lemma}

\begin{remark}
Note that in $(2)$ and $(3)$ above, $1.9$ and $1.1$ are of nothing special. In fact, we may choose any numbers.
\end{remark}

\begin{proof} The proof consists of two steps.

{\it Step 1:} We first show that there exists a constant $\delta=\delta(n,\lam,K,\a,\rho)>0$ such that for any point $x_0\in \M$ and constant $r_0>0$, there exists a function $\psi\in C^{\infty}(\M)$ satisfying:\\
(a) $0\leq \psi<1$ everywhere.\\
(b) $\psi>\delta$ in $B(x_0,\delta r_0)$.\\
(c) $\psi=0$ in $\M\setminus B(x_0,r_0)$.\\
(d) $|\d \psi|<r_0^{-1}$ and $|\Delta \psi|<r_0^{-2}$ everywhere.\\

Consider the function $G(x)=G(x,Ar_0^2;x_0,0)$, where $G(x,t;x_0,0)$ denotes the heat kernel, and $A$ is a constant to be determined. From Theorem \ref{thmheatkernel}, we have
\be\label{G(x0)}
C_1(Ar_0^2)^{-\frac{n}{2}}\leq G(x_0)\leq C_3(Ar_0^2)^{-\frac{n}{2}}.
\ee
If $d(x,x_0)\geq r_0$, then
\[
G(x)\leq C_3(Ar_0^2)^{-\frac{n}{2}}e^{-\frac{d^2(x,x_0)}{C_4Ar_0^2}}\leq C_3(Ar_0^2)^{-\frac{n}{2}}e^{-\frac{1}{AC_4}}.
\]
Thus by choosing $A=\frac{1}{C_4\ln(2C_3/C_1)}$, we have
\[
G(x)\leq\frac{1}{2}C_1(Ar_0^2)^{-\frac{n}{2}},\ in\  \M\setminus B(x_0,r_0).
\]

Also in $B(x_0,r_0)$,
\be\label{cutoff eq1}
G(x)\leq G(x_0)+|\d G|d(x,x_0)\leq C_3\left[(Ar_0^2)^{-\frac{n}{2}}+(Ar_0^2)^{-\frac{n+1}{2}}r_0\right]=(1+A^{-1/2})C_3(Ar_0^2)^{-\frac{n}{2}}.
\ee
Let
\[\hat{\psi}(x)=\max\{G(x)-0.6C_1(Ar_0^2)^{-\frac{n}{2}},0\}.\]
Then $supp\, \hat{\psi}\subset\subset B(x_0, r_0)$. Moreover, when $\hat{\psi}>0$, \eqref{cutoff eq1} implies that
\be\label{cutoff eq2}
\hat{\psi}(x)\leq [(1+A^{-1/2})C_3-0.6C_1](Ar_0^2)^{-\frac{n}{2}},
\ee
in $B(x_0,r_0)$, and Theorem \ref{thmheatkernel} gives
\be\label{cutoff eq3}
|\d \hat{\psi}|=|\d G(x)|\leq C_3(Ar_0^2)^{-\frac{n+1}{2}},\ and\ |\Delta \hat{\psi}|=|\Delta G|=|\partial_t G|\leq C_3(Ar_0^2)^{-\frac{n+2}{2}}.
\ee

On the other hand, from \eqref{G(x0)}, we have $\hat{\psi}(x_0)\geq 0.4C_1(Ar_0^2)^{-\frac{n}{2}}$. Thus, setting $\hat{\delta}=\frac{0.4C_1}{(1+A^{-1/2}+A^{-1})C_3}$ yields that in $B(x_0,\hat{\delta} r_0)$,
\be\label{cutoff eq4}
\al
\hat{\psi}(x)\geq \hat{\psi}(x_0)-|\d \hat{\psi}|(x')d(x,x_0)\geq & 0.4C_1(Ar_0^2)^{-\frac{n}{2}}-C_3(Ar_0^2)^{-\frac{n+1}{2}}\hat{\delta} r_0\\
=&  \hat{\delta}[(1+A^{-1})C_3](Ar_0^2)^{-\frac{n}{2}}.
\eal
\ee

Combining \eqref{cutoff eq2}, \eqref{cutoff eq3} and \eqref{cutoff eq4}, we see that function $\tilde{\psi}=\frac{1}{(1+A^{-1})C_3}(Ar_0^2)^{\frac{n}{2}}\hat{\psi}$ satisfies
\[
0\leq \tilde{\psi}\leq 1\ in\ B(x_0, r_0),\quad \tilde{\psi}\geq \hat{\delta}\ in\ B(x_0,\hat{\delta} r_0),
\]
\[
|\d \tilde{\psi}|\leq r_0^{-1},\ and\ |\Delta \tilde{\psi}|\leq r_0^{-2}.
\]

Now, let $\psi=\tilde{\psi}^3$. Then $\psi$ is $C^2$ differentiable and still satisfies (a)-(d).\\

{\it Step 2:} From step 1 and using the same covering argument as in \cite{CN} and also Lemma 6.13 in \cite{Bam}, we can prove the existence of the cut-off function $\phi$ in the lemma.  For a fixed small $\epsilon>0$,  let $N$ be the smallest integer such that $N$ balls of radius
$\e \delta r_0$ cover $B(x_0, 1.1 r_0)$.  We denote these balls by $B(p_i, \e \delta r_0)$, $i=1, ...N$.
By the volume comparison theorem and noncollapsing condition, $N$ is uniformly bounded.  Let $\psi_i$ be the cut-off function
obtained in step1 for the ball $B(p_i, \e r_0)$. Let $\eta$ be a smooth function from $[0, \infty)$ to $[0, 1]$ such that $\eta(s)=1$
when $s \ge 1$. Then it is easy to check the function
\[
\psi = \eta\left(\Sigma^N_{i=1} \delta^{-1} \psi_i \right)
\]is the desired cut-off function when $\e<0.1$.
\end{proof}
\medskip

\section{Segment inequality, excess estimate and $\e$-splitting map}

In this section, we will build the three essential tools: segment inequality, excess estimate and $\e$-splitting map for proving certain degeneration
results for Riemannian manifolds under the modified Ricci curvature bound.
From now on, unless otherwise stated , we always
assume that \eqref{riccond}, \eqref{condition V} and \eqref{noncollapsing-0}
hold, i.e.,
\be
\label{basichypo1}
\Ric+\frac{1}{2}\mc{L}_{V} g\geq - \lam g, \quad
|V|(y)\leq \frac{K}{d(y,O)^\a}, \quad \a \in [0, 1),
\ee
and
\be\label{basichypo2}
 \vol (B(y, 1)) \ge \rho>0, \quad \forall y \in \M.
\ee

Note that this condition allows gradient Ricci solitons whose potential function
 $L$  has bounded gradient.  It also covers the case when the Bakry-\'Emery Ricci curvature has a
lower bound and the potential function $L$ has bounded gradient.
In this case,  F. Wang and X.H. Zhu \cite{WZ} have obtained, among other things,
 degeneration results using
weighted manifold approach.   If both $|\nabla L|$ and
$\Delta L$ are bounded,  then, according to Tian and Z.L. Zhang \cite{TZ}, one can use a conformal transform to convert the metric to a new one with lower Ricci bound. Then the degeneration results follow from Cheeger and Colding's work.

In proving the degeneration results, we will follow the road maps by Cheeger and Colding.

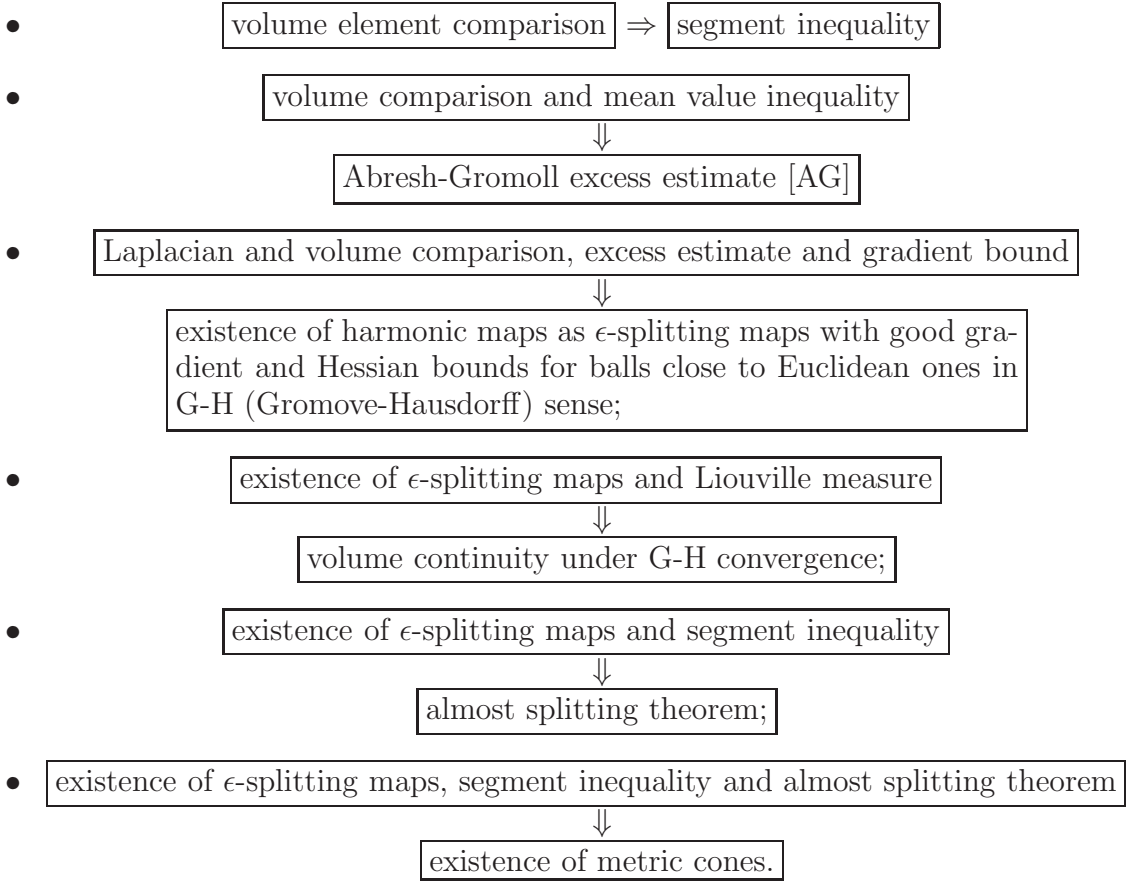
\begin{itemize}

\item {\hfil \framebox{volume element comparison} $\Rightarrow$ \framebox{segment inequality}}
\medskip

\item {\hfil \framebox{volume comparison and mean value inequality}}

{\hfil $\Downarrow$}

{\hfil \framebox{Abresh-Gromoll excess estimate \cite{AG}}}
\medskip

\item {\hfil \framebox{Laplacian and volume comparison, excess estimate and gradient bound}}

     {\hfil $\Downarrow$}

     {\hfil \fbox{\parbox{0.7\textwidth}{existence of harmonic maps as $\e$-splitting maps with good gradient and Hessian bounds for balls close to Euclidean ones in G-H (Gromove-Hausdorff) sense;}}}
\medskip

\item {\hfil \fbox{existence of  $\e$-splitting maps  and Liouville measure}}

 {\hfil $\Downarrow$}

 {\hfil \fbox{volume continuity under G-H convergence;}}
\medskip

\item {\hfil \fbox{existence of  $\e$-splitting maps and segment inequality}}

 {\hfil $\Downarrow$}

 {\hfil \fbox{almost splitting theorem;}}
\medskip

\item {\hfil \fbox{existence of  $\e$-splitting maps, segment inequality and almost splitting theorem}}

     {\hfil $\Downarrow$}

     {\hfil \fbox{existence of metric cones.}}\\

\end{itemize}

Identical proofs will be omitted. Proof that needs moderate modifications will be presented in the appendices. Proofs that requires new ingredients will be presented in the main text.

Given a smooth function $f\geq 0$, let
\be\label{definition F}
\mathcal{F}_f(x_1,x_2)=\inf_{\gamma}\int_0^{d(x_1,x_2)} f(\gamma(s))ds,
\ee
where the infimum is taken over all minimal geodesics from $x_1$ to $x_2$.

We first prove a similar segment inequality as in \cite{ChCo1}. The proof is essentially the same as the proof of Theorem 2.11 in \cite{ChCo1}, where the key ingredient in the proof is the comparison of the volume elements, which in our case is \eqref{volume element}. For completeness, we present the proof in Appendix II. From Theorem \ref{thm volume element comparison}, we see that \eqref{volume element} only requires assumption \eqref{basichypo1}. More explicitly, we have

\begin{theorem}[Segment inequality]
\lab{thsegineq}
Let $F_f(x_1,x_2)$ be the function defined in \eqref{definition F}. Suppose that \eqref{basichypo1} is satisfied. Then for any $r>0$ and measurable sets
$A_1, A_2\subseteq B(x,r)$, we have
\be\label{segment}
\int_{A_1\times A_2} \mathcal{F}_f(x_1,x_2)\leq C(n,\lam,K,\a)(\vol(A_1)+\vol(A_2))\, r \int_{B(x,3r)}f dg.
\ee
\end{theorem}

The proof can be found in  Appendix II.

Next, we deal with the excess estimate. Fix $q_+$ and $q_-$ in $\M$, and define the excess function on $\M$ as
\[
e(y)=d(y,q_+)+d(y,q_-)-d(q_+,q_-).
\]
For any $x\in \M$, if we denote by
\[
b_{\pm}(y)=d(y,q_{\pm})-d(x,q_{\pm}),
\]
then
\be\label{b and e}
b_+(y)+b_-(y)=e(y)-e(x).\ee

Since we have the heat kernel bounds,  the following mean value type inequality can be obtained routinely by using Duhamel's formula.  One may imitate the proof of Corollary 4.8 in \cite{ZZ}, which is a version of Theorem 8.18 in \cite{GT} and Lemma 2.1 of Colding and Naber \cite{CN}.
\begin{lemma}\label{mean value theorem}
Let $u(x)$ be a nonnegative function satisfying
\begin{equation}
\Delta u\leq\xi(x),
\end{equation}
where $\xi(x)\geq0$ is a smooth function.

If \eqref{basichypo1} and \eqref{basichypo2} are satisfied, then for any $q>\frac{n}{2}$, there exists a constant $C=C(n,\lam,K,\a,\rho)$ such that
\begin{equation}
\oint_{B(x,r)}u(y)dy\leq C\left(u(x)+r^{2}\|\xi\|^*_{q,B(x,r)}\right)
\end{equation}
holds for any $x\in M$ and $0<r\leq 1$.

More generally, we have
\begin{equation}
\oint_{B(x,r)}u(y)dy\leq C\left(\inf_{B(x,r)}u(\cdot)+r^{2}\|\xi\|^*_{q,B(x,r)}\right).
\end{equation}
\end{lemma}

Applying the above lemma to the excess function yields the following excess estimate. In the context below, we will use $\Psi(\e_1,\e_2,\cdots,\e_k| c_1,c_2,\cdots,c_N)$ to denote a nonnengative function such that for any fixed $c_1,c_2,\cdots,c_N$,
\[
\lim_{\e_1,\e_2,\cdots,\e_k\rightarrow 0}\Psi=0.
\]

\begin{theorem}[Excess estimate]
\label{excess}
Suppose that \eqref{basichypo1} and \eqref{basichypo2} are satisfied. Assume in addition that \be\label{excess assumption 1}
d(x,q_{\pm})= O( \lam^{-\frac{1}{2}})
\ee
and
\be\label{excess assumption 2}
e(x)\leq \e R,
\ee
for some $R\leq 1$. Then there exists a $\Psi=\Psi(\lam,K,\e|n,R,\a,\rho)$ such that
\begin{equation}\label{excess estimate}
e(y)\leq \Psi R,\ in\  B(x,R).
\end{equation}
\end{theorem}

\begin{proof} The proof, which is different from the original one in \cite{AG}, is inspired by Remark 2.9 \cite{CN}.
Note that by \eqref{b and e},
\[
\Delta e(y)=\Delta b_++\Delta b_-\leq \frac{n-1}{d(y,q_+)}+\frac{n-1}{d(y,q_-)}+\psi_++\psi_-,
\]
where by Proposition \ref{prop laplacian comparison},
\be\label{eq psi}
\psi_{\pm}=\left(\Delta d(y,q_{\pm})-\frac{n-1}{d(y,q_{\pm})}\right)_+\leq \frac{\lam}{3}d(y,q_{\pm})+|V|(y)+\frac{C(\a)K}{d^{\a}(y,q_{\pm})}.\ee
According to \eqref{eq psi}, Lemma \ref{lemma Lq bound of V} and Lemma \ref{mean value theorem}, for $y\in B(x,R)$, we have
\[
\al
\oint_{B(x,R)}e(y)dy\leq& C\left(\inf_{B(x,R)}u(\cdot)+\frac{2(n-1)R^2}{\lam^{-\frac{1}{2}}-R}+R^{2}||\psi_+||^*_{q,B(x,R)}+R^{2}||\psi_-||^*_{q,B(x,R)}\right)\\
\leq & C\left(\e R+\frac{2(n-1)R^2}{\lam^{-\frac{1}{2}}-R}+\lam R^2(\lam^{-\frac{1}{2}}+R)+\frac{KR^2}{(\lam^{-\frac{1}{2}}-R)^\a}+R^{2-\a}K\right)\\
=&\Psi_1 R.
\eal
\]
Therefore, by using Theorem \ref{thm volume element comparison}, for any $y\in B(x,(1-\Psi_2)R)$, we have
\[\al
\oint_{B(y,\Psi_2 R)}e(z)dz\leq&\frac{\v(B(x,R))}{\v(B(y,\Psi_2 R))}\oint_{B(x,R)}e(z)dz\\
\leq&\frac{\v(B(y,2R))}{\v(B(y,\Psi_2 R))}\Psi_1 R\\
\leq&\frac{\Psi_1}{\Psi_2^n}R.
\eal\]
Hence, if we choose $\Psi_2=\Psi_1^{\frac{1}{n+1}}$, then the above inequality becomes
\[
\oint_{B(y,\Psi_2 R)}e(z)dz\leq \Psi_2 R.
\]
It implies that there exists a $z\in B(y,\Psi_2 R)$ such that
\[
e(z)\leq \Psi_2 R,
\]
and hence by mean value theorem
\[
e(y)\leq e(z)+|\d e|d(y,z)\leq \Psi_2 R+2\Psi_2 R=3\Psi_2 R.
\]
Since $y$ is an arbitrary point in $B(x,(1-\Psi_2)R)$, applying mean value theorem one more time for points in $B(x,R)\setminus B(x,(1-\Psi_2)R)$ completes the proof of the theorem.
\end{proof}

Let $h_{+}$ and $h_-$ be harmonic functions in $B(x,R)$ such that $h_{\pm}\huge|_{\pa B(x,R)}=b_{\pm}$. Since the excess estimate \eqref{excess estimate} has been established, following the method in section 6 of \cite{ChCo1} (see also section 9 in \cite{Ch}), it is not hard to show that

\begin{lemma}\label{harmonic approximation}
 Under the basic assumption \eqref{basichypo1}, \eqref{basichypo2}, suppose \eqref{excess assumption 1} and \eqref{excess assumption 2} hold i.e. $d(x,q_{\pm})= O( \lam^{-\frac{1}{2}})$
and
$e(x)\leq \e R$. Then there exists an $r_0=r_0(n,\lam, K,\a, \rho)$ such that for any $R\leq r_0$ and some $\Psi=\Psi(\lam, K,\e|n,R,\a,\rho)$, we have\\
(1) $\displaystyle |h_{\pm}-b_{\pm}|\leq \Psi R$ in $B(x,R)$;\\
(2) $\displaystyle \oint_{B(x,R)}|\d b_{\pm}-\d h_{\pm}|^2\leq \Psi$;\\
(3) $\displaystyle \oint_{B(x,\frac{1}{2}R)}|\d^2 h_{\pm}|^2\leq \Psi R^{-2}$.
\end{lemma}

\hspace{-.5cm}{\it Proof of (1):} Since $\Delta h_{\pm}=0$, by Proposition \ref{prop laplacian comparison}, we have
\[
\Delta (h_{\pm}-b_{\pm})\geq -\psi_{\pm},
\]
where $\psi_{\pm}=\frac{\lam}{3}d(\cdot,q_{\pm})+|V|(\cdot)+\frac{C(\a)K}{d(\cdot,q_{\pm})^{\a}}$.

Thus, Theorem \ref{thm maximum principle} implies that
\be\label{ea 1}
\al
h_{\pm}-b_{\pm}\leq & C(n)R^2||\psi_{\pm}||^*_{q,B(x,R)}\\
\leq& C(n,\lam,K,\a,\rho)R^2\left[\lam(\lam^{-\frac{1}{2}}+R)+ KR^{-\a}+K(\lam^{-\frac{1}{2}}-R)^{-\a}\right]\\
=&\Psi R.
\eal
\ee

For the lower bounds, notice that $b_+(y)+b_-(y)=e(y)-e(x)$. It follows from Theorem \ref{excess} that
\be\label{ea 2}
-\e R\leq b_+(y)+b_-(y)\leq \Psi R,
\ee
which in turn, by the maximal principle, implies that
\be\label{ea 3}
h_+(y)+h_-(y)\geq-\e R.
\ee
Combining \eqref{ea 1}, \eqref{ea 2} and \eqref{ea 3}, we derive
\[
b_+\leq -b_-+\Psi R\leq -h_-+2\Psi R\leq h_++2\Psi R+\e R,
\]
i.e.,
\[
h_+-b_+\geq -\Psi R.
\]
A lower bound for $h_--b_-$ can be obtained in a similar way.\\

\hspace{-.5cm}{\it Proof of (2):} We only prove the result for $h_+$, the proof for $h_-$ is similar.
By (1), Theorem \ref{thm volume element comparison} and \eqref{ea 1}, we get
\[
\al
\oint_{B(x,R)}|\d h_{+}-\d b_+|^2=&\oint_{B(x,R)}(h_+-b_+)\Delta b_+\\
\leq&  \Psi R\left|\oint_{B(x,R)}(\Delta b_+)_- - (\Delta b_+)_+\right|+2\Psi R\oint_{B(x,R)}(\Delta b_+)_+\\
\leq& \left(\frac{\v(\pa B(x,R))}{\v(B(x,R))}\right)\Psi R +2\Psi R\oint_{B(x,R)}\psi_+\\
\leq& \Psi .\\
\eal
\]

\hspace{-.5cm}{\it Proof of (3):} Again, we only present the proof for $h_+$.
By Lemma \ref{lem cutoff}, we may choose a cut-off function $\phi$ so that
\[supp\,\phi\subset B(x,R),\quad \phi|_{B(x,R/2)}=1,\quad |\d \phi|^2+|\Delta \phi|\leq \frac{C(n,\lam,K,\a,\rho)}{R^2}.\]
By the Bochner's formula,
\be\label{ap 1}
\al
\oint_{B(x,R)}2\phi|\d^2 h_+|^2=&\oint_{B(x,R)}\phi\left(\Delta|\d h_+|^2-2Ric(\d h_+,\d h_+)\right)\\
\leq&\oint_{B(x,R)}\phi\left(\Delta(|\d h_+|^2-1)+2\d_iV_j\d_ih_+\d_jh_++2\lam |\d h_+|^2\right).
\eal
\ee
Notice that
\[
\al
\oint_{B(x,R)}2\phi\d_iV_j\d_ih_+\d_jh_+=&\oint_{B(x,R)}-2\d_i\phi V_j\d_ih_+\d_jh_+ - 2\phi V_j\d_i h_+\d_{i}\d_jh_+\\
\leq& \oint_{B(x,R)}2|\d \phi||V||\d h_+|^2 + \oint_{B(x,R)}\phi|\d^2h_+|^2+\phi|V|^2|\d h_+|^2.
\eal
\]
Thus, \eqref{ap 1} becomes
\be\label{ap 2}
\oint_{B(x,R)}\phi|\d^2 h_+|^2\leq \oint_{B(x,R)}(|\d h_+|^2-1)\Delta\phi+2|\d \phi||V||\d h_+|^2+\phi|V|^2|\d h_+|^2+2\lam |\d h_+|^2.
\ee
By the gradient estimate in Theorem \ref{thmpoisgrad} and the assumed bound on $V$, we see that
\be\label{ap 3}
|\d h_+|^2\leq CR^{-2}(||h_+||^*_{2,B(x,2R)})^2\leq CR^{-2}(|||b_+|+\Psi R||^*_{2,B(x,2R)})^2\leq C(n,
\lam, K,\a,\rho, R).
\ee
Then, using the conclusion in (2), Lemma \ref{lemma Lq bound of V} and \eqref{ap 3} for the RHS of \eqref{ap 2} yields
\[
\oint_{B(x,R)}\phi|\d^2 h_+|^2\leq \Psi R^{-2},
\]
which finishes the proof of (3) due to the fact that $\phi|_{B(x,\frac{1}{2}R)}=1$ and the volume comparison theorem \ref{thm volume element comparison}.\qed\\

Lemma \ref{harmonic approximation} together with the segment inequality implies, as in Theorem 6.62 \cite{ChCo1},

\begin{theorem}
If \eqref{basichypo1}, \eqref{basichypo2}, \eqref{excess assumption 1} and \eqref{excess assumption 2} hold, then for some $r_0=r_0(n,\lam, K,\a, \rho)$ and any $R\leq r_0$, there exists a metric space $X$ such that for some ball $B((0,\tilde{x}),\frac{1}{4}R)\subseteq \mathbb{R}\times X$, with the product metric , we have
\[
d_{GH}\left(B(x,\frac{1}{4}R), B((0,\tilde{x}),\frac{1}{4}R)\right)\leq \Psi.
\]

\end{theorem}

The above theorem amounts to the following almost splitting theorem, as in Theorem 6.64 \cite{ChCo1}.

\begin{theorem}[Almost splitting]
Let $(\M_i,d_i)\xrightarrow{d_{GH}}(Y,d)$, where $d_i$ is the distance function induced by the Riemaniann metric $g_i$ on $\M$. Suppose that
\[
Ric_{\M_i}+\frac{1}{2}\mathcal{L}_{V_i}g_{\M_i}\geq -\lam_i g_{\M_i},\ |V_i|(y_i)\leq \frac{K_i}{d^\a_{i}(y_i,O_i)},\ and\ \vol(B(y_i,1))\geq \rho,
\]
where $O_i$ is a fixed point in $\M_i$, $y_i$ is an arbitrary point in $\M_i$, and $\lam_i, K_i\rightarrow 0$ as $i\rightarrow\infty$.

If $Y$ contains a line, then $Y$ splits as an isometric product, $\mathbb{R}\times X$ for some metric space $X$.
\end{theorem}

Finally, Lemma \ref{harmonic approximation} also yields, according to \cite{Co} (see also \cite{ChCo1}), the existence of a so called $\e$-splitting map (see also Definition 1.6 and Lemma 1.7 in \cite{ChNa}).

\begin{definition}
\lab{defhsplit}
A map $h=(h_1, h_2,\cdots, h_k):\ B(x,r)\rightarrow \mathbb{R}^k$ is an $\e$-splitting map, if\\
(1) Each $h_i$, $i=1,\cdots,k$, is a harmonic function;\\
(2) $|\d h|\leq 1+\e$;\\
(3) $\displaystyle \oint_{B(x,r)}\left|<\d h_i, \d h_j>-\delta_{ij}\right|^2\leq \e^2$, $\forall i,j$;\\
(4) $\displaystyle r^2\oint_{B(x,r)}|\d^2 h_i|\leq \e^2$, $\forall i$.
\end{definition}

\begin{lemma}\label{lem coordinate approximation}
Suppose that \eqref{basichypo1} and $\eqref{basichypo2}$ are satisfied, and
\[
d_{GH}(B(x,\lam^{-1}), \underline{B}(0,\lam^{-1}))\leq \e.
\]
Then for any $R\leq 1$, there exists a $\Psi$-splitting map $h=(h_1, h_2, \cdots, h_n):\ B(x,R)\rightarrow \mathbb{R}^n$ for some $\Psi=\Psi(\lam, K,\e| n,R,\a,\rho)$. Moreover, in the Gromov-Hausdorff sence, $h_i$ and $\pi_i$ are $\Psi$ close.

Here $\underline{B}(0,L)$ denotes the ball centered at the origin with radius $L$ in $\mathbb{R}^n$, and $\pi_i$ is the $i$-th coordinate function on $\mathbb{R}^n$.
\end{lemma}

\begin{remark} Here we recall the idea of the proof of the lemma in the papers cited above.
If $d_{GH}(B(x,\lam^{-1}), \underline{B}(0,\lam^{-1}))\leq \e$, then there exists a $2\e$ isometry $f:\ \underline{B}(0,\lam^{-1})\rightarrow B(x,\lam^{-1})$. Thus, for any $y\in B(x,\lam^{-1})$, there is a point $\underline{y}\in \underline{B}(0,\lam^{-1})$ so that $d(f(\underline{y}),y)\leq 2\e$. From this $f$, one can easily show that a map formed by
certain Buseman functions is a $4\e$ isometry. Then one can use harmonic functions with the same boundary value on a ball as the Buseman functions to form a $6 \e$ isometry.  Then Lemma \ref{harmonic approximation} shows  this harmonic map satisfies
the conditions in  definition \ref{defhsplit}. Hence it is an $\e$ splitting.

In the above lemma, $h_i$ and $\pi_i$ being $\Psi$ close in Gromov-Hausdorff sense means that for any $y\in B(x,\lam^{-1})$, we have $\left|h_i(y)-\pi_i(\underline{y})\right|\leq \Psi$.
\end{remark}

\begin{remark}   If $\a=0$ in condition  \eqref{basichypo1}, then the noncollapsing condition \eqref{basichypo2} is not necessary
for all conclusions in this section.
\end{remark}

\medskip
\section{Main theorems:  volume continuity, metric cone,  size of singular set}

In this section, we present the volume convergence and volume continuity theorems and prove that any tangent cone of a Gromov-Hausdorff limit under conditions \eqref{basichypo1} and \eqref{basichypo2} is a metric cone.

First, we state the volume convergence and volume continuity theorems. Under the assumption that Ricci curvature bounded from below, the volume continuity was first proved by Colding \cite{Co}. Here we adapt a similar argument as in Theorems 9.31 and 9.40 in \cite{Ch}, which essentially uses Lemma \ref{lem coordinate approximation} and the segment inequality.  With the preparation in the
previous section, the remaining proof now is the same as in \cite{Co} (See also Theorems 9.31 and 9.40 in \cite{Ch}).

\begin{theorem}[Volume convergence]
Suppose that \eqref{basichypo1} and $\eqref{basichypo2}$ are satisfied, and
\[
d_{GH}(B(x,R), \underline{B}(0,R))\leq \e,
\]
for some $R\leq 1$ and sufficiently small $\e>0$. Then for some $\Psi=\Psi(\lam,K,\e | n,R,\a,\rho)$ we have
\[
\vol(B(x,R))\geq(1-\Psi)\vol(\underline{B}(0,R)).
\]
Here $\underline{B}(0,R)$ is the ball in $\mathbb{R}^n$.
\end{theorem}

The above theorem implies the volume continuity under Gromov-Hausdorff convergence.

\begin{theorem}[Volume continuity]
For a sequence of Riemannian manifolds $\{(\M_i,g_i)\}$, assume that
\[
Ric_{\M_i}+\frac{1}{2}\mathcal{L}_{V_i}g_{\M_i}\geq - \lam g_{\M_i},\ |V_i|(y_i)\leq \frac{K}{d^\a_{i}(y_i,O_i)},\ and\ \vol(B(y_i,1))\geq \rho,
\]
where $O_i$ is a fixed point in $\M_i$ and $y_i$ is an arbitrary point in $\M_i$.

Let $(\M,g)$ be a Riemannian manifold. If $(\M_i,d_i,x_i)\xrightarrow{d_{GH}}(\M,d,x)$ in pointed Gromov-Hausdorff topology, then for all $r<\infty$,
\[
\vol(B(x_i,r))=\vol(B(x,r)).\\
\]
\end{theorem}

Next, we show a cone rigidity theorem in the spirit of \cite{ChCo1}. Although our proof follows the idea there, certain modifications and extensions in the choice of the axillary functions are needed.

For a metric space $(Z,d)$, denote by $(C(Z),\hat{d})$ the metric cone on $Z$. That is, $C(Z)=(0,\infty)\times Z$ with metric:
\[
\hat{d}\left((r_1, z_1), (r_2, z_2)\right)=\left\{\begin{array}{cc} r_1^2+r_2^2-2r_1r_2\cos(d(z_1,z_2)),&\ if\ d(z_1,z_2)\leq \pi,\\
r_1+r_2,&\ if\ d(z_1,z_2)>\pi.\end{array}\right.
\]

In the following, unless otherwise specified, $C_i$ will denote constants depending on $n,\lam,K,\a,$ and   $\rho$.

\begin{theorem}[Cone rigidity]\label{thm cone rigidity}
Suppose that \eqref{basichypo1} and \eqref{basichypo2} are satisfied. There exists a $r_0=r_0(n,\lam,K,\a,\rho)$ such that if for $R\leq r_0$ and $\delta>0$,
\be\label{volume condition}
(1-\delta)\vol(B(x,R))\leq \frac{R}{n}\vol(\pa B(x,R)),
\ee
then for some $\Psi=\Psi(\lam,K,\delta|n,R,\a,\rho)$, we have
\be\label{GH close}
d_{GH}(B(x,R), C_{0,R}(Z))\leq \Psi,
\ee
where $Z$ is a length space with
\be\label{diam Z}
diam(Z)\leq \pi+\Psi.
\ee
\end{theorem}
\begin{proof}

$Step\ 1$.   We show that the metric is close to the Hessian of a function in average sense.

First, we derive bounds on $f$. In \cite{ChCo1} (see also \cite{Ch}), the lower bound of $f$ was obtained by constructing certain comparison function $\underline{L}_R$ and using the maximum principle on $f-\underline{L}_{3R}(r_y)$, where $y$ is an arbitrary point in $\M$ with $d(y,x)=2R$ and $r_y$ is the distance function from $y$. The same method can be applied to drive the lower bound of $f$ in this case as well.
 However, here we will use Dirichlet Green's function for simplicity.

For any open subset $\Omega$ in $\M$, denote by $\Gamma_{\Omega}(z,y)$ the Dirichlet Green's function in $\Omega$. That is,
\[
\Delta_x \Gamma_{\Omega}(z,y)=-\delta_y;\ \Gamma_{\Omega}(z,y)=\Gamma_{\Omega}(y,z);\ \Gamma_{\Omega}(\cdot,y)\huge|_{\pa \Omega}=0.
\]
Given $x\in \M$, let $f$ be the smooth solution of the following Dirichlet problem:
\[
\Delta f=1\ in\ B(x,R);\ f\huge|_{\pa B(x,R)}=\frac{R^2}{2n}.
\]
By using the Green's function, we can write $f$ as
\be\label{mean value f}
f(z)=\frac{R^2}{2n}-\int_{B(x,R)}\Gamma_{B(x,R)}(z,y)dg(y).
\ee
By the upper bound of the heat kernel in Theorem \ref{thmheatkernel}, we have
\be\label{bound Gamma}
\al
\Gamma_{B(x,R)}(z,y)&\leq  \int_0^{1}G(z,t;y,0)dt +\int_1^{\infty}C e^{-c t} dt\\
&\leq\int_0^{1}C_3t^{-\frac{n}{2}}e^{-\frac{d^2(z,y)}{C_4t}}dt + C
\leq C_5d^{-(n-2)}(z,y).
\eal
\ee Here we have used the large time exponential decay of the Dirichlet heat kernel on a ball.
Note that this decay can be obtained by proving that the $L^2$ norm decays exponentially first.
This needs the
help of the Poincar\'e inequality on a ball, which is true due the the heat kernel bound and gradient
estimate. Then one can apply the mean value inequality.
Therefore, \eqref{mean value f} and \eqref{bound Gamma} imply
\be\label{f lower bound}\al
f(z)\geq\frac{R^2}{2n}-\int_{B(x,R)}C_5d^{-(n-2)}(z,y)dg(y)
\geq\frac{R^2}{2n}-C(n)e^{C(\a)KR^{1-\a}+4\lam R^2}R^{2}.
\eal\ee
This gives a lower bound of $f$.

For an upper bound of $f$, let $r_x$ be the distance function from $x$. Then
\be\label{Delta f-U}
\al
\Delta f-\Delta\frac{r_x^2}{2n}= 1-\frac{r_x}{n}\Delta r_x-\frac{1}{n}
\geq -\frac{r_x}{n}\psi,
\eal
\ee
where $\psi=\left(\Delta r_x-\frac{n-1}{r_x}\right)_+$.
From Proposition \ref{prop laplacian comparison}, one has
\[
\psi\leq \frac{\lam}{3} r_x+|V|+\frac{C(\a)K}{r_x^\a}.
\]
Thus,
\[
r_x\psi \leq r_x\left(\frac{\lam}{3}r_x+|V|+\frac{C(\a)K}{r_x^\a}\right)\leq \lam R^2+R|V|+C(\a)KR^{1-\a},
\]
and \eqref{Delta f-U} becomes
\be\label{Delta f-U 1}
\Delta f-\Delta\frac{r_x^2}{2n}\geq -C(n)\left(\lam R^2+C(\a)KR^{1-\a}\right)-C(n)R|V|.
\ee
According to Theorem \ref{thm maximum principle}, we get
\be\label{f upper bound}
\al
f(z)\leq & \frac{r_x^2(z)}{2n}+C(n)R^2\left(\lam R^2+C(\a)KR^{1-\a}\right)+C(n)R^3||V||^*_{q,B(x,R)}\\
\leq & \frac{r_x^2(z)}{2n}+C(n,\lam,K,\a,\rho)\left(\lam R^4+KR^{3-\a}\right).
\eal
\ee

Next, we show that $f$ and $\frac{r_x^2}{2n}$ are close. By the volume condition \eqref{volume condition} in the theorem,
\be\label{integral upper bound Delta f-U}
\al
\int_{B(x,R)}\Delta f-\Delta \frac{r_x^2}{2n}=&\v(B(x,R))-\int_{\partial B(x,R)}<\d \frac{r_x^2}{2n}, \d r_x> dS\\
=& \v(B(x,R))-\frac{R}{n}\v(\pa B(x,R))\\
\leq & \delta \v(B(x,R)).
\eal
\ee
From \eqref{Delta f-U 1}, we can also get the lower bound
\be\label{integral lower bound Delta f-U}
\al
\oint_{B(x,R)}\Delta f-\Delta \frac{r_x^2}{2n}\geq& -\oint_{B(x,R)}C(n)\left(\lam R^2+C(\a)KR^{1-\a}\right)+C(n)R|V|\\
\geq& -C(n,\lam,K,\a,\rho)(\lam R^2+KR^{1-\a}).
\eal
\ee
Therefore, we have
\[-C(n,\lam,K,\a,\rho)(\lam R^2+KR^{1-\a})\leq \oint_{B(x,R)}\Delta f-\Delta \frac{r_x^2}{2n}\leq \delta.\]

Also, in \eqref{f upper bound}, setting $A:=C(n,\lam,K,\a,\rho)(\lam R^4+KR^{3-\a})$, the constant on the right hand side, yields
\[
\frac{r_x^2}{2n}-f+A\geq0.
\]
Combining the inequality above and \eqref{f lower bound}, \eqref{Delta f-U}, \eqref{integral upper bound Delta f-U} and Lemma \ref{lemma Lq bound of V}, we have
\[
\al
0\leq&\oint_{B(x,R)}\left[\Delta (f-\frac{r_x^2}{2n})+\frac{r_x}{n}\psi\right]\left[\frac{r_x^2}{2n}-f+A\right]\\
= &\oint_{B(x,R)} A\Delta (f-\frac{r_x^2}{2n})+(\frac{r_x^2}{2n}-f)\Delta (f-\frac{r_x^2}{2n})+\frac{r_x}{n}\psi\left[\frac{r_x^2}{2n}-f+A\right]\\
\leq& \delta A-\oint_{B(x,R)} |\d(f-\frac{r_x^2}{2n})|^2+\frac{1}{n}\left[C(n)e^{C(\a)KR^{1-\a}+4\lam R^2}R^{2}+A\right]\left(\lam R^2+C(\a)KR^{1-\a}\right).
\eal
\]
It follows that
\be\label{W2 bound U-f}
\al
\oint_{B(x,R)} |\d(f-\frac{r_x^2}{2n})|^2\leq&\delta A+\frac{1}{n}\left[C(n)e^{C(\a)KR^{1-\a}+4\lam R^2}R^{2}+A\right]\left(\lam R^2+C(\a)KR^{1-\a}\right)\\
=& \Psi(\delta,\lam,K|n,R,\a,\rho)
\eal
\ee
Therefore, by the Poincar\'e inequality, we have
\be\label{L2 bound U-f}
\oint_{B(x,R)}(\frac{r_x^2}{2n}-f)^2\leq \Psi(\delta,\lam,K|n,R,\a,\rho).
\ee
From \eqref{L2 bound U-f}, volume comparison theorem \eqref{thm volume element comparison}, the gradient estimate \eqref{thmpoisgrad}, and the bounds of $f$ (\eqref{f lower bound} and \eqref{f upper bound}), it is not hard to see that
\be\label{bound f-U}
\left|\frac{r_x^2}{2n}-f\right|\leq \Psi_1(\delta,\lam,K|n,R,\a,\rho), \ on\ B(x,(1-\Psi_1)R),
\ee
where $\Psi_1=\Psi^{1/(n+1)}$.

Finally, we show the closeness of $\d^2 f$ and the metric $g$ in the average sense. Note that, by Theorems \ref{thmpoisgrad} and \ref{thm maximum principle}, we have $|\d f|\leq C$ in $B(x,(1-\Psi_1)R)$. This together with \eqref{bound f-U} and \eqref{W2 bound U-f} gives
\be\label{L1 bound df-dU}
\oint_{B(x,(1-\Psi_1)R)}\left||\d f|^2-\frac{2f}{n}\right|\leq \Psi_1.
\ee
Now, let $\phi$ be the cut-off function as in Lemma \ref{lem cutoff} with $\phi\big|_{B(x,(1-\Psi_1)R)}=1$. From \eqref{L1 bound df-dU} and Bochner's formula, we have
\be\label{Hess f 1}
\al
CR^{-2}\Psi_1\geq& \frac{1}{2}\oint_{B(x,R)}\Delta \phi\cdot(|\d f|^2-\frac{2f}{n})\\
\geq&\frac{1}{2}\oint_{B(x,R)}\phi\cdot\Delta(|\d f|^2-\frac{2f}{n})\\
=&\oint_{B(x,R)}\phi\left(|\d^2 f|^2+Ric(\d f,\d f)-\frac{1}{n}\right)\\
\geq&\oint_{B(x,R)}\phi\left(|\d^2 f|^2-\frac{1}{n}-\lam|\d f|^2-\d_iV_jf_if_j\right).
\eal
\ee
Using Lemma \ref{lemma Lq bound of V} and the boundedness of $|\d f|$, we obtain
\be\label{Hess f 2}
\al
-\oint_{B(x,R)}\phi\d_iV_jf_if_j=&\oint_{B(x,R)}\phi_iV_jf_if_j+\phi V_jf_j+\phi V_jf_if_{ij}\\
\geq&-CR^{-1}\oint_{B(x,R)}|V|-\left(\oint_{B(x,R)}\phi |V|^2|\d f|^2\right)^{1/2}\cdot\left(\oint_{B(x,R)}\phi|\d^2 f|^2\right)^{1/2}\\
\geq&-CR^{-1-\a}K-\frac{1}{4}K^{-1/2}\oint_{B(x,R)}\phi |V|^2|\d f|^2-K^{1/2}\oint_{B(x,R)}\phi|\d^2 f|^2\\
\geq&-CR^{-1-\a}K-CR^{-2\a}K^{3/2}-K^{1/2}\oint_{B(x,R)}\phi|\d^2 f|^2.
\eal
\ee
It then follows from \eqref{Hess f 1} and \eqref{Hess f 2} that
\[
\al
\Psi_1
\geq& -C(\lam R^2+R^{1-\a}K+R^{2-2\a}K^{3/2})+CR^2\oint_{B(x,R)}\phi(1-K^{1/2})\left[|\d^2 f|^2-\frac{1}{n}\right]-K^{1/2}\frac{1}{n}.
\eal
\]
That is
\[
\al
\oint_{B(x,R)}\phi\left[|\d^2 f|^2-\frac{1}{n}\right]
\leq& (1-K^{1/2})^{-1}\left[\Psi_1+C(\lam R^2+R^{1-\a}K+R^{2-2\a}K^{3/2}+K^{1/2}R^2)\right]\\
=&\Psi_2.
\eal
\]
Therefore,
\be\label{bound Hess f}
\al
\oint_{B(x,(1-\Psi_1)R)}\left|\d^2f-\frac{1}{n}g_{ij}\right|^2\leq\oint_{B(x,R)}\phi\left[|\d^2 f|^2-\frac{1}{n}\right]\leq \Psi_2.
\eal
\ee

$Step\ 2.$ Now we may follow the argument in section 3 and 4 in \cite{ChCo1} (see also Theorem 9.45 of \cite{Ch}) to show \eqref{GH close} and \eqref{diam Z}.

Roughly speaking, for any points $y,z,w\in B(x,(1-\Psi_1)R)$, with $r_x(y)=r_x(z)=a$ and $z$ being the point on $r_x^{-1}(a)$ closest to $w$, by using the segment inequality \eqref{segment}, \eqref{L1 bound df-dU} and \eqref{bound Hess f}, one can find $y^*, z^*, w^*$ close to $y,z,w$, respectively, and
\be\label{cone eq1}
\int_{r_x(z^*)}^{r_x(z^*)+d(z^*,w^*)}|\d^2 f|(\gamma_s(t))dtds\leq \Psi,
\ee
and
\be\label{cone eq2}
\int_{r_x(z^*)}^{r_x(z^*)+d(z^*,w^*)}\left||\d f|^2-\frac{2f}{n}\right|(\sigma(s))ds\leq \Psi.
\ee
Here $\sigma(s)$, $r_x(z^*)\leq s\leq r_x(z^*)+d(z^*,w^*)$, is the minimal geodesic from $z^*$ to $w^*$, and $\gamma_s:\ [0,l(s)]\rightarrow \M$ the minimal geodesic from $x^*$ to $\sigma(s)$.

From \eqref{cone eq1}, \eqref{cone eq2} and the first variation of arc length, it then follows that $d(x^*,w^*)$ is close to the distance between them when they are considered to be two points in a metric cone on $r_x^{-1}(a)$.
\end{proof}

As pointed out in \cite{Ch}, when applying Theorem \ref{thm cone rigidity}, it is often more convenient  to verify the following condition instead of \eqref{volume condition}.


\begin{lemma}\label{lem cone rigidity}
If
\[
(1-\delta')\frac{\vol(B(x,\frac{1}{2}R))}{\vol(\underline{B}(0,\frac{1}{2}R))}\leq \frac{\vol(B(x,R))}{\vol(\underline{B}(0,R))},
\]
then for any $0<\eta\leq\frac{1}{2}$ and some $\Psi=\Psi(\delta',\lam,K|n,\a,\rho)$, we have
\[
(1-\Psi)\frac{\vol(B(x,(1-\eta)R))}{\vol(\underline{B}(0,(1-\eta)R))}\leq \frac{\vol(\pa B(x,(1-\eta)R))}{\vol(\pa \underline{B}(0,(1-\eta)R))},
\]
i.e.,
\[
(1-\Psi)\vol(B(x,(1-\eta)R))\leq \frac{R}{n}\vol(\pa B(x,(1-\eta)R)).
\]
\end{lemma}

\proof See appendix II.\qed\\

From Lemma \ref{lem cone rigidity} and Theorem \ref{thm cone rigidity}, one can follow a similar argument as in Theorem 9.69 in \cite{Ch} to show that

\begin{theorem}
\lab{thvolclosetoghclose}
Suppose that \eqref{basichypo1} and \eqref{basichypo2} are satisfied. There exists an $r_0=r_0(n,\lam,K,\a,\rho)$ such that if
\[
\vol(B(x,R))\geq(1-\e)\vol(\underline{B}(0,R)),
\]
for some $R\leq r_0$ and sufficiently small $\e>0$, then
\[
d_{GH}(B(x,R),\underline{B}(0,R))\leq \Psi(\lam,K,\e | n,R,\a,\rho).
\]
\end{theorem}

Theorem
\ref{thm cone rigidity} also implies that the tangent cones of a Gromov-Hausdorff limit are metric cones as shown in Theorem 5.2 in \cite{ChCo2}.

\begin{theorem}[Tangent cones are metric cones]
\label{thtcmc}
Let $(\M_i, g_i)$ be a sequence of Riemannian manifolds satisfying \eqref{basichypo1} and $\eqref{basichypo2}$ for some $\lam,\ K,\ \a,$ and $\rho$. Suppose that $(\M_i,d_i)\xrightarrow{d_{GH}}(Y,d)$ in Gromov-Hausdorff topology. Then any tangent cone of $Y$ is a metric cone.
\end{theorem}

Finally, following the proof of Theorem 6.1 in \cite{ChCo2} closely, we deduce

\begin{theorem}[Size estimate of the singular set]
\label{thsizesc}
Let $(\M_i, g_i)$ be a sequence of Riemannian manifolds satisfying \eqref{basichypo1} and \eqref{basichypo2} for some $\lam,\ K,\ \a,$ and $\rho$. Suppose that $(\M_i,d_i)\xrightarrow{d_{GH}}(Y,d)$ in Gromov-Hausdorff topology. Then $dim S(Y) \le n-2$. Here, $S(Y)$ is the singular set of $Y$ and $dim$ stands for the Hausdorff dimension.
\end{theorem}  Recall that $S(Y)$ is the set of points in $Y$ such that a tangent cone at $y$
is not isometric to a Euclidean space.

 To prove the theorem, one just needs to show that no tangent cone is isometric to the half Euclidean space.
We mention that a small modification is needed in the proof. Namely, as pointed out in
the proof of Theorem 6.1 \cite{ChCo2}, one can replace the usage of Perelman's
theorem in \cite{Pe} by using Theorem A.1.8 in \cite{ChCo2}. The conclusion of this later theorem continues to hold in
our situation since it follows from the conclusions of Theorem A.1.1 in  and
Theorem A.1.5 in \cite{ChCo2}.
But Theorem A.1.1 in \cite{ChCo2} is a version of Reifenberg's theorem \cite{Re} which is
a result for more general metric spaces. So it is still valid here.  The conclusion of Theorem A.1.5
is still true here due to Theorem \ref{thvolclosetoghclose} above.
So the map $\hat f$ in the proof of Theorem 6.1 \cite{ChCo2} is an imbedding. Therefore it is
mod 2 degree is $1$. But from p437 in \cite{ChCo2},
 its mod 2 degree is also $0$. This contradiction finishes the proof.

\begin{remark}

In the case of Bakry-\'Emery Ricci curvature condition, the conclusions of the theorems  in this section  continue to hold under conditions
\eqref{gradient riccond} and \eqref{condition L}.  Namely, one only needs the potential function $L$ to be H\"older continuous
and its gradient in certain $L^p$ space. In this case it is  not clear whether the segment inequality \ref{thsegineq} still hold.
However, a weaker form of it is available from the work \cite{TZ2}, which is sufficient for the proof of the main results.

\end{remark}

\begin{remark}
With a little more effort, along the lines of Section 3 in \cite{PeWe1}, one can also prove some compactness and finiteness results of topological types under  conditions \eqref{basichypo1} and \eqref{basichypo2}  and extra
integral assumption on the curvature tensor. Also, with two sided bounds on the Ricci curvature with two vector fields $U$ and $V$, namely,
\[
\frac{1}{2}\mc{L}_{U} g  +  \lam g \ge \Ric \ge - \frac{1}{2}\mc{L}_{V} g  - \lam g; \quad |U|+
|V|(y)\leq \frac{K}{d(y,O)^\a},\ \a \in [0, 1);
\quad \vol (B(x, 1)) \ge \rho,
\]
one can also prove that $S(Y)$ is a closed set using the technique in \cite{ChCo2} and the results here.
\end{remark}

\section{Appendix I}

In this section, we give a proof of Theorem \ref{thmpoisgrad}, namely
\begin{theorem}
Assume that
\[
\Ric+\frac{1}{2}\mc{L}_{V} g\geq - \lam g, \quad
|V|(y)\leq \frac{K}{d(y,O)^\a}, \quad \a \in [0, 1),
\quad \vol (B(x, 1)) \ge \rho.
\]
Then there exists a positive constant  $r_0=r_0(n,\lam,K,\a,\rho)\leq 1$ such that,  for any $x\in\M$, $0<r\leq r_0$ and smooth functions $u$ and $f$ satisfying the equation
\[
\Delta u=f,\ in\  B(x,r),
\]
we have
\[
\sup_{B(x,\frac{1}{2}r)}|\d u|^2\leq C(n,\lam,K,\a,\rho)r^{-2}\left[(||u||^*_{2,B(x,r)})^2+(||f||^*_{2q,B(x,r)})^2\right],
\]
for any $q>n/2$.  Moreover
\[
\sup_{B(x,\frac{1}{2}r)} u^2\leq C(n,\lam,K,\a,\rho) \left[(||u||^*_{2,B(x,r)})^2+(||f||^*_{q ,B(x,r)})^2\right].
\]

In particular if $\a=0$, then the conclusions hold without the non-collapsing condition.
\end{theorem}

\begin{proof}
Let
\[
v=|\d u|^2+||f^2||^*_{q, B(x,r)}.
\]
Then
\be\label{Delta v}
\al
\Delta v&=2|\d^2 u|^2+2<\d \Delta u,\d u>+2Ric(\d u, \d u)\\
&\geq 2u_if_i-2\lam v-(\mc{L}_{V}g)_{ij}u_iu_j.
\eal
\ee
For any $p>0$,
\be\label{Delta v p}
\al
\Delta v^p=& pv^{p-1}\Delta v+p(p-1)v^{p-2}|\d v|^2\\
\geq& 2pv^{p-1}u_if_i-2\lam pv^p-pv^{p-1}(\mc{L}_{V}g)_{ij}u_iu_j+\frac{p-1}{p}v^{-p}|\d v^p|^2.
\eal
\ee
Let $B=B(x,r)$ and $p\geq 1$, then it follows from \eqref{Delta v p} that
\be\label{GE eq1}
\al
\int_B |\d(\eta v^p)|^2=&\int_B |\eta\d v^p+v^p\d \eta|^2\\
=& \int_B \eta^2|\d v^p|^2+v^{2p}|\d\eta|^2+2\eta v^p<\d v^p, \d \eta>\\
=& \int_B v^{2p}|\d\eta|^2-\eta^2v^p\Delta v^p\\
\leq& \int_B v^{2p}|\d\eta|^2 -2p\eta^2v^{2p-1}u_if_i+2\lam p\eta^2v^{2p}+p\eta^2v^{2p-1}(\mc{L}_{V}g)_{ij}u_iu_j.
\eal
\ee
Moreover, since
\[
\al
&\int_B \eta^2v^{2p-1}(\mc{L}_{V}g)_{ij}u_iu_j\\
=&\int_B \eta^2v^{2p-1}\d_jV_iu_iu_j\\
=&-\int_B 2\eta v^{2p-1}\eta_jV_iu_iu_j+(2p-1)\eta^2v^{2p-2}v_jV_iu_iu_j+\eta^2v^{2p-1}V_iu_{ij}u_j+\eta^2v^{2p-1}fV_iu_i\\
\leq& \int_B v^{2p}|\d \eta|^2 +\eta^2 v^{2p-2}|V|^2|\d u|^4 -\frac{2p-1}{p}\eta v^{p-1}V_iu_iu_j[(\eta v^p)_j-v^p\eta_j]-\frac{1}{2}\eta^2v^{2p-1}V_iv_i\\
&\quad +\frac{1}{2}\eta^2v^{2p-2}f^2|\d u|^2+\frac{1}{2}\eta^2v^{2p}|V|^2\\
\leq& \int_B v^{2p}|\d \eta|^2 +\frac{3}{2}\eta^2 v^{2p}|V|^2 -\frac{2p-1}{p}\eta v^{p-1}V_iu_iu_j[(\eta v^p)_j-v^p\eta_j]\\
&\quad -\frac{1}{2p}\eta^2v^{p}V_i[(\eta v^p)_i-v^p\eta_i]+\frac{1}{2}\eta^2v^{2p-2}f^2|\d u|^2\\
\eal
\]
\[
\al
\leq& \int_B v^{2p}|\d \eta|^2 +\frac{3}{2}\eta^2 v^{2p}|V|^2 +\frac{1}{4p}|\d(\eta v^p)|^2 + \frac{(2p-1)^2}{p}\eta^2 v^{2p}|V|^2+ \frac{2p-1}{2p}v^{2p}|\d \eta|^2\\
&\quad+\frac{2p-1}{2p}\eta^2v^{2p}|V|^2+\frac{1}{4p}|\d(\eta v^p)|^2+ \frac{1}{4p}\eta^2v^{2p}|V|^2+\frac{1}{4p}\eta^2v^{2p}|V|^2\\
&\quad +\frac{1}{4p}v^{2p}|\d\eta|^2+\frac{1}{2}\eta^2v^{2p-1}f^2\\
=&\int_B \frac{8p-1}{4p}v^{2p}|\d \eta|^2 +\frac{2(2p-1)^2+5p}{2p}\eta^2 v^{2p}|V|^2
+\frac{1}{2p}|\d(\eta v^p)|^2+\frac{1}{2}\eta^2v^{2p-1}f^2,
\eal
\]
we can rewrite \eqref{GE eq1} as
\be\label{GE eq2}
\al
\int_B |\d(\eta v^p)|^2
\leq& \int_B 2v^{2p}|\d\eta|^2 -4p\eta^2v^{2p-1}u_if_i+4\lam p\eta^2v^{2p}+\frac{8p-1}{2}v^{2p}|\d \eta|^2\\
&\quad +(2(2p-1)^2+5p)\eta^2 v^{2p}|V|^2+p\eta^2v^{2p-1}f^2.
\eal
\ee

Notice that
\[
\al
&-\int_B\eta^2v^{2p-1}u_if_i\\
=&\int_B\eta^2v^{2p-1}f^2+2\eta v^{2p-1}f\d_iu\d_i\eta+(2p-1)\eta^2v^{2p-2}f\d_iu\d_iv\\
=&\int_B\eta^2v^{2p-1}f^2+2\eta v^{2p-1}f\d_iu\d_i\eta+\frac{2p-1}{p}\eta v^{p-1}f\d_iu(\d_i(\eta v^p)-v^p\d_i\eta)\\
=&\int_B\eta^2v^{2p-1}f^2+\frac{1}{p}\eta v^{2p-1}f\d_iu\d_i\eta+\frac{2p-1}{p}\eta v^{p-1}f\d_iu\d_i(\eta v^p)\\
\leq& \int_B \eta^2v^{2p-1}f^2+\frac{1}{2p}\eta^2v^{2p-2}f^2|\d u|^2+\frac{1}{2p}v^{2p}|\d \eta|^2+\frac{1}{8p}|\d(\eta v^p)|^2 +\frac{2(2p-1)^2}{p}\eta^2v^{2p-2}f^2|\d u|^2\\
\leq& \int_B \frac{4(2p-1)^2+1}{2p}\eta^2v^{2p-1}f^2+\frac{1}{2p}v^{2p}|\d \eta|^2+\frac{1}{8p}|\d(\eta v^p)|^2.
\eal
\]
Thus, it follows that
\be\label{GE eq3}
\al
\int_B |\d(\eta v^p)|^2
\leq& \int_B 4v^{2p}|\d\eta|^2+8\lam p\eta^2v^{2p}+(8p-1)v^{2p}|\d \eta|^2+(4(2p-1)^2+10p)\eta^2 v^{2p}|V|^2\\
&\quad+2p\eta^2v^{2p-1}f^2+(16(2p-1)^2+4)\eta^2v^{2p-1}f^2+4v^{2p}|\d \eta|^2,
\eal
\ee
Assume that $r_i=(\frac{1}{2}+\frac{1}{2^{i+2}})r$, $i=0,1,2,\cdots$. Construct cut-off function $\phi_i(s)$ such that
\[
supp\,\phi_i\subseteq[0,r_i],\quad \phi_i=1\  on\  [0,r_{i+1}],\ and\ -\frac{52^i}{r}\leq \phi_i'\leq 0;
\]

Let $\eta_i(y)=\phi_i(d(y,x))$. Then \eqref{GE eq3} implies
\be\label{GE eq4}
\al
\oint_{B(x,r_i)} |\d(\eta_i v^p)|^2
\leq& \oint_{B(x,r_i)} 8\lam p\eta_i^2v^{2p}+16pv^{2p}|\d \eta_i|^2+30p^2\eta_i^2 v^{2p}|V|^2 +70p^2\eta_i^2v^{2p-1}f^2.
\eal
\ee

On the other hand, since $\frac{r}{2}\leq r_i\leq \frac{3}{4}r$, by the volume comparison theorem
\be\label{young}
\al
&p^2\oint_{B(x,r_i)}\eta_i^2 v^{2p-1}f^2\\
\leq&\frac{p^2}{||f^2||^*_{q, B(x,r)}}\oint_{B(x,r_i)}\eta_i^2 v^{2p}f^2\\
\leq& C(n,\lam,K,\a,\rho)p^2\left(\oint_{B(x,r_i)}(\eta_i v^p)^{\frac{2q}{q-1}}\right)^{\frac{q-1}{q}}\\
\leq& C(n,\lam,K,\a,\rho)p^2\left(\oint_{B(x,r_i)}(\eta_i v^p)^{a\cdot\frac{2q}{q-1}\cdot b}\right)^{\frac{q-1}{qb}}\cdot\left(\oint_{B(x,r_i)}(\eta_i v^p)^{(1-a)\cdot\frac{2q}{q-1}\cdot\frac{b}{b-1}}\right)^{\frac{(q-1)(b-1)}{qb}}\\
\leq& \e\left(\oint_{B(x,r_i)}(\eta_i v^p)^{a\cdot\frac{2q}{q-1}\cdot b}\right)^{\frac{q-1}{qba}}+\e^{-\frac{a}{1-a}}C^{\frac{1}{1-a}}p^{\frac{2}{(1-a)}}\left(\oint_{B(x,r_i)}(\eta_i v^p)^{(1-a)\cdot\frac{2q}{q-1}\cdot\frac{b}{b-1}}\right)^{\frac{(q-1)(b-1)}{qb(1-a)}}.
\eal
\ee
Here in the last step, we have used Young's inequality
\[
xy\leq \e x^{\gamma}+\e^{-\frac{\gamma^*}{\gamma}}y^{\gamma^*},\ \forall x,y>0, \gamma>1, \frac{1}{\gamma}+\frac{1}{\gamma^*}=1,
\]
for $\gamma=\frac{1}{a}$.

Since $q>\frac{n}{2}$, if we choose $a=\frac{n}{2q}$, and $b=\frac{2q-2}{n-2}$, it follows from \eqref{young} that
\[
\al
p^2\oint_{B(x,r_i)}\eta_i^2 v^{2p-1}f^2
\leq  \e\left(\oint_{B(x,r_i)}(\eta_i v^p)^{\frac{2n}{n-2}}\right)^{\frac{n-2}{n}}+\e^{-\frac{a}{1-a}}C^{\frac{2q}{2q-n}}p^{\frac{4q}{2q-n}}\oint_{B(x,r_i)}\eta_i^2 v^{2p}.
\eal
\]
Hence, \eqref{GE eq4} becomes
\be\label{GE eq5}
\al
\oint_{B(x,r_i)} |\d(\eta_i v^p)|^2
\leq& \oint_{B(x,r_i)} 8\lam p\eta_i^2v^{2p}+16pv^{2p}|\d \eta_i|^2+30p^2\eta_i^2 v^{2p}|V|^2\\
& +70\e\left(\oint_{B(x,r_i)}(\eta_i v^p)^{\frac{2n}{n-2}}\right)^{\frac{n-2}{n}}+70\e^{-\frac{a}{1-a}}C^{\frac{2q}{2q-n}}p^{\frac{4q}{2q-n}}\oint_{B(x,r_i)}\eta_i^2 v^{2p}.
\eal
\ee
If $q\in(\frac{n}{2},\frac{n}{2\a})$, then
\[
\al
&30p^2\oint_{B(x,r_i)}\eta_i^2v^{2p}|V|^2\\
\leq& 30p^2\left(\oint_{B(x,r_i)}(\eta_i v^p)^{\frac{2q}{q-1}}\right)^{\frac{q-1}{q}}\cdot\left(\oint_{B(x,r_i)}|V|^{2q}\right)^{1/q}\\
\leq& p^2C(n,\lam,K,\a,\rho)r_i^{-2\a}\left(\oint_{B(x,r_i)}(\eta_i v^p)^{\frac{2q}{q-1}}\right)^{\frac{q-1}{q}}\\
\leq& \e r_i^{-2\a}\left(\oint_{B(x,r_i)}(\eta_i v^p)^{2n/(n-2)}\right)^{\frac{n-2}{n}}+\e^{-\frac{a}{1-a}}p^{\frac{2}{1-a}}C^{\frac{1}{1-a}}r_i^{-2\a}\oint_{B(x,r_i)}\eta_i^2 v^{2p},
\eal
\]
for any $\e>0$ and $a=\frac{n}{2q}$. Here, the last step above is followed similarly as in \eqref{young}.
Therefore, \eqref{GE eq5} becomes
\be\label{GE eq6}
\al
\oint_{B(x,r_i)} |\d(\eta_i v^p)|^2
\leq& \oint_{B(x,r_i)} 8\lam p\eta_i^2v^{2p}+16pv^{2p}|\d \eta_i|^2\\
& +\e r_i^{-2\a}\left(\oint_{B(x,r_i)}(\eta_i v^p)^{\frac{2n}{n-2}}\right)^{\frac{n-2}{n}}+\e^{-\frac{a}{1-a}}C^{\frac{2q}{2q-n}}p^{\frac{4q}{2q-n}}r_i^{-2\a}\oint_{B(x,r_i)}\eta_i^2 v^{2p}\\
& +70\e\left(\oint_{B(x,r_i)}(\eta_i v^p)^{\frac{2n}{n-2}}\right)^{\frac{n-2}{n}}+70\e^{-\frac{a}{1-a}}C^{\frac{2q}{2q-n}}p^{\frac{4q}{2q-n}}\oint_{B(x,r_i)}\eta_i^2 v^{2p}.
\eal
\ee
By the Sobolev inequality \eqref{L2 Sobolev} and \eqref{GE eq6},
\be\label{GE eq7}
\al
&\left(\oint_{B(x,r_i)}(\eta_i v^p)^{\frac{2n}{n-2}}\right)^{(n-2)/n}\\
\leq& C(n)r_i^{2}\oint_{B(x,r_i)}|\d(\eta_iv^p)|^2\\
\leq& C(n)r_i^{2}\oint_{B(x,r_i)} 8\lam p\eta_i^2v^{2p}+16pv^{2p}|\d \eta_i|^2\\
&\quad +C(n)\e r_i^{2-2\a}\left(\oint_{B}(\eta v^p)^{\frac{2n}{n-2}}\right)^{\frac{n-2}{n}}+C(n)\e^{-\frac{a}{1-a}}C^{\frac{2q}{2q-n}}p^{\frac{4q}{2q-n}}r_i^{2-2\a}\oint_{B(x,r_i)}\eta_i^2 v^{2p}\\
&\quad +C(n)r_i^2\e\left(\oint_{B(x,r_i)}(\eta_i v^p)^{\frac{2n}{n-2}}\right)^{\frac{n-2}{n}}+C(n)r_i^2\e^{-\frac{a}{1-a}}p^{\frac{4q}{2q-n}}\oint_{B(x,r_i)}\eta_i^2 v^{2p}.
\eal
\ee
Since $r_i\leq r\leq 1$ and $\a<1$, we may choose $\e=\e(n)$ small so that the above inequality becomes
\be\label{GE eq8}
\al
\left(\oint_{B(x,r_i)}(\eta_i v^p)^{\frac{2n}{n-2}}\right)^{\frac{n-2}{n}}
\leq& C(n,\lam,K,\a,\rho)r_i^{2}\oint_{B(x,r_i)} pv^{2p}|\d \eta_i|^2+ p^{\frac{4q}{2q-n}}\eta_i^2 v^{2p}.
\eal
\ee

From the volume comparison theorem, we have
\[
\al
\left(\oint_{B(x,r_{i+1})}(v^p)^{2n/(n-2)}\right)^{(n-2)/n}
\leq& C(n,\lam,K,\a,\rho)\left(\oint_{B(x,r_i)}(\eta_i v^p)^{2n/(n-2)}\right)^{(n-2)/n}\\
\leq& C(n,\lam,K,\a,\rho)\oint_{B(x,r_i)} 2^{2i}pv^{2p}+ p^{\frac{4q}{2q-n}} v^{2p}.
\eal
\]

Now let $\mu=\frac{n}{n-2}$ and choose $p=\frac{1}{2}\mu^i$ for $i=0,1,2,\cdots$. Then
\[
\al
\left(\oint_{B(x,r_{i+1})}v^{\mu^{i+1}}\right)^{(n-2)/n}=&\left(\oint_{B(x,r_{i+1})}(v^p)^{2n/(n-2)}\right)^{(n-2)/n}\\
\leq& C(n,\lam,K,\a,\rho)(2^{2i-1}\mu^i+\mu^{2qi/(2q-n)})\oint_{B(x,r_i)}v^{\mu^i}\\
\leq& C(n,\lam,K,\a,\rho)4^{2qi/(2q-n)}\oint_{B(x,r_i)}v^{\mu^i},
\eal
\]
i.e.,
\be\label{GE eq9}
\al
||v||^*_{\mu^{i+1},B(x,r_{i+1})}\leq C^{\mu^{-i}}4^{\frac{2q}{2q-n}i\mu^{-i}}||v||^*_{\mu^{i},B(x,r_i)}
\eal
\ee
By using \eqref{GE eq9} iteratively, we get
\be\label{GE eq10}
\sup_{B(x,\frac{1}{2}r)}v\leq C^{\sum\mu^{-i}}4^{\frac{2q}{2q-n}\sum i\mu^{-i}}||v||^*_{1,B(x,\frac{3}{4}r)}\leq C(n,\lam,K,\a,\rho)||v||^*_{1,B(x,\frac{3}{4}r)}.
\ee

Choose $\eta=\phi(d(x,\cdot))$, where $\phi$ satisfies
\[
supp\, \phi\subset[0,r],\ \phi_l\equiv1\ in\ [0,\frac{3}{4}r],\ |\phi'|\leq 5r^{-1}.
\]
Since
\[
\al
\int_{B(x,r)}\eta^2|\d u|^2&=\int_{B(x,r)}-\eta^2uf-2\eta u\d_i u\d_i \eta\\
&\leq \int_{B(x,r)}\frac{1}{2}u^2\eta^2+\frac{1}{2}f^2\eta^2+\frac{1}{2}\eta^2|\d u|^2+2u^2|\d \eta|^2,
\eal
\]
this together with the definition of $\eta$ imply that
\[
\al
\oint_{B(x,r)}\eta^2|\d u|^2&\leq 4\oint_{B(x,r)}u^2\eta^2+f^2\eta^2+u^2|\d \eta|^2\\
&\leq 100r^{-2}(||u||^*_{2,B(x,r)})^2+4||f^2||^*_{q,B(x,r)}.
\eal
\]
From above, we arrive at
\be\label{GE eq11}
\al
||v||^*_{1,B(x,\frac{3}{4}r)}&\leq \frac{\v(B(x,r))}{\v(B(x,\frac{3}{4}r))}\oint_{B(x,r)}\eta^2(|\d u|^2+||f^2||^*_{q,B(x,r)})\\
&\leq C(n,\lam,K,\a,\rho)r^{-2}\left[(||u||^*_{2,B(x,r)})^2+(||f||^*_{2q,B(x,r)})^2\right].
\eal
\ee
Combining \eqref{GE eq10} and \eqref{GE eq11}, we deduce that, for any $q\in(n,\frac{n}{\a})$,
\[
\sup_{B(x,\frac{1}{2}r)}|\d u|^2\leq ||v||_{\infty,B(x,\frac{1}{2}r)}\leq C(n,\lam,K,\a,\rho)r^{-2}\left[(||u||^*_{2,B(x,r)})^2+(||f||^*_{2q,B(x,r)})^2\right].
\]
This finishes the proof of the theorem due to the fact that $||f||^*_{q,B(x,r)}\leq ||f||^*_{q',B(x,r)}$ whenever $q'\geq q$.

If $\a=0$, then the Sobolev inequality in section 3 holds without the volume non-collapsing condition. Therefore the conclusion of the theorem also holds without it.
\end{proof}

\section{Appendix II}

Here we give a proof of Theorem \ref{thsegineq}, the segment inequality.
\begin{proof}[Proof of Theorem \ref{thsegineq}]

We may assume that the minimal geodesic from $x_1$ to $x_2$ is unique, since the set of this type of points $(x_1,x_2)$ is dense. Moreover, since
\[
\al
\mathcal{F}_f(x_1,x_2)=&\int_0^{\frac{d(x_1,x_2)}{2}}f(\gamma(s))ds+ \int_{\frac{d(x_1,x_2)}{2}}^{d(x_1,x_2)}f(\gamma(s))ds\\
=&\int_{\frac{d(x_1,x_2)}{2}}^{d(x_1,x_2)}f(\gamma(d(x_1,x_2)-t))dt + \int_{\frac{d(x_1,x_2)}{2}}^{d(x_1,x_2)}f(\gamma(s))ds\\
:=&\mathcal{F}_{f,1}(x_1,x_2) + \mathcal{F}_{f,2}(x_1,x_2),
\eal
\]
to estimate $\mathcal{F}_f$ it suffices to estimate $\mathcal{F}_{f,2}(x_1,x_2)$.

Now fixe $x_1\in A_1$. For any unit tangent vector $\theta$ at $x_1$, let $\gamma(0)=x_1$ and $\gamma'(0)=\theta$. Let $l(\theta)$ be the largest number $\leq 2r$ such that $\gamma\big|_{[0,l(\theta)]}$ is minimal. Then
\[
\al
\int_{A_2}\mathcal{F}_{f,2}(x_1,x_2)dg(x_2)\leq& \int_{B(x_1,2r)}\mathcal{F}_{f,2}(x_1,x_2)dg(x_2)\\ =&\int_{S^{n-1}}\int_0^{l(\theta)}\mathcal{F}_{f,2}(x_1,\gamma(s))w(s,\theta)dsd\theta\\
=&\int_{S^{n-1}}\int_0^{l(\theta)}\left[\int_{\frac{s}{2}}^{s}f(\gamma(u))du\right]w(s,\theta)dsd\theta\\
\leq&\int_{S^{n-1}}\int_0^{l(\theta)}\int_{\frac{s}{2}}^{s}f(\gamma(u))w(u,\theta)C(n,\lam,K,\a)\frac{s^{n-1}}{u^{n-1}}dudsd\theta\\
\leq&C(n,\lam,K,\a)\int_{S^{n-1}}\int_0^{2r}\int_{0}^{2r}f(\gamma(u))w(u,\theta)dudsd\theta\\
=&C(n,\lam,K,\a)r\int_{B(x_1,2r)}f(x)dg\\
\leq&C(n,\lam,K,\a)r\int_{B(x,3r)}f(x)dg.
\eal
\]
In the fourth step above, we have used \eqref{volume element} to get
\[
w(s,\theta)\leq C(n,\lam,K,\a)\frac{s^{n-1}}{u^{n-1}}w(u,\theta).
\]
Thus,
\[
\int_{A_1}\int_{A_2}\mathcal{F}_{f,2}(x_1,x_2)dg(x_2)dg(x_1)\leq C(n,\lam,K,\a)r\v(A_1)\int_{B(x,3r)}f(x)dg.
\]
Similarly, we can prove
\[
\int_{A_2}\int_{A_1}\mathcal{F}_{f,1}(x_1,x_2)dg(x_1)dg(x_2)\leq C(n,\lam,K,\a)r\v(A_2)\int_{B(x,3r)}f(x)dg.
\]
This finishes the proof of the theorem.
\end{proof}

Next, we prove Lemma \ref{lem cone rigidity}. First  from Theorem \ref{thm volume element comparison}, we can make the following

\begin{claim}
Suppose that \eqref{basichypo1} and \eqref{basichypo2} hold. Let $r\leq 1$, and $0<\eta<1$. Then for some $\Psi=\Psi(\lam, K|n,\a,\rho)$, we have
\be\label{comp1}
\frac{\vol(\pa B(x,r))}{\vol(\pa \underline{B}(0,r))}\leq (1+\Psi)\frac{\vol(A_{\eta r,r}(x))}{\vol(\underline{A}_{\eta r, r}(0))},
\ee
and
\be\label{comp2}
\frac{\vol(A_{\eta r,r}(x))}{\vol(\underline{A}_{\eta r, r}(0))}\leq (1+\Psi)\frac{\vol(\pa B(x,\eta r))}{\vol(\pa \underline{B}(0,\eta r))}.
\ee Here $A_{r_1, r_2}(x)=B(x, r_2)-B(x, r_1)$ is the annulus in $\M$  and $\underline{A}_{r_1, r_2}(0)$ is the corresponding Euclidean annulus.
\end{claim}

\begin{proof}
Given $r\leq 1$, for any $s<r$, by Theorem \ref{thm volume element comparison}, we have
\be\label{comp eq 1}
\frac{w(r,\cdot)}{r^{n-1}}\leq (1+\Psi)\frac{w(s,\cdot)}{s^{n-1}}.
\ee
It follows that
\be\label{comp eq 2}
s^{n-1}\int_{S^{n-1}}w(r,\theta)d\theta\leq (1+\Psi)r^{n-1}\int_{S^{n-1}}w(s,\theta)d\theta.
\ee
Integrating both sides from $\eta r$ to $r$ with respect to $s$ gives \eqref{comp1}.

Similarly, one can show \eqref{comp2}. This proves the claim.
\end{proof}

Now we are ready to prove Lemma \ref{lem cone rigidity}, which is
\begin{lemma}
If
\[
(1-\delta')\frac{\vol(B(x,\frac{1}{2}R))}{\vol(\underline{B}(0,\frac{1}{2}R))}\leq \frac{\vol(B(x,R))}{\vol(\underline{B}(0,R))},
\]
then for any $0<\eta\leq\frac{1}{2}$ and some $\Psi=\Psi(\delta',\lam,K|n,\a,\rho)$, we have
\[
(1-\Psi)\frac{\vol(B(x,(1-\eta)R))}{\vol(\underline{B}(0,(1-\eta)R))}\leq \frac{\vol(\pa B(x,(1-\eta)R))}{\vol(\pa \underline{B}(0,(1-\eta)R))}.
\]
\end{lemma}

\begin{proof}
By Theorem \ref{thm volume element comparison} and the assumption of the lemma, we have
\[
(1-\delta')\frac{\v(B(x,(1-\eta)R))}{\v(\underline{B}(0,(1-\eta)R))}\leq(1-\delta')(1+\Psi)\frac{\vol(B(x,\frac{1}{2}R))}{\vol(\underline{B}(0,\frac{1}{2}R))}\leq (1+\Psi)\frac{\vol(B(x,R))}{\vol(\underline{B}(0,R))}.
\]
It implies that
\[
(1-\delta')\v(B(x,(1-\eta)R))\vol(\underline{B}(0,R))\leq (1+\Psi)\vol(B(x,R))\v(\underline{B}(0,(1-\eta)R)).
\]
Subtracting $(1-\delta')\v(B(x,(1-\eta)R))\v(\underline{B}(0,(1-\eta)R))$ from both sides above yields
\[
\al
&(1-\delta')\v(B(x,(1-\eta)R))\vol(\underline{A}_{(1-\eta)R,R}(0))\\
\leq& \vol(A_{(1-\eta)R,R}(x))\v(\underline{B}(0,(1-\eta)R))+\Psi\vol(B(x,R))\v(\underline{B}(0,(1-\eta)R))\\
&\  +\delta'\v(B(x,(1-\eta)R))\v(\underline{B}(0,(1-\eta)R)).
\eal
\]
Dividing both sides above by $\v(\underline{B}(0,(1-\eta)R))\vol(\underline{A}_{(1-\eta)R,R}(0))$ and using Theorem \ref{thm volume element comparison} give
\[\al
&(1-\delta')\frac{\v(B(x,(1-\eta)R))}{\v(\underline{B}(0,(1-\eta)R))}\\
\leq& \frac{\vol(A_{(1-\eta)R,R}(x))}{\vol(\underline{A}_{(1-\eta)R,R}(0))}+\Psi\frac{\vol(B(x,R))}{\vol(\underline{A}_{(1-\eta)R,R}(0))}+\delta'\frac{\v(B(x,(1-\eta)R))}{\vol(\underline{A}_{(1-\eta)R,R}(0))}\\
\leq& \frac{\vol(A_{(1-\eta)R,R}(x))}{\vol(\underline{A}_{(1-\eta)R,R}(0))}+\left[2^n\Psi(1+\Psi)+\delta'\right]\frac{\vol(B(x,(1-\eta)R))}{\vol(\underline{A}_{(1-\eta)R,R}(0))}.
\eal
\]
Moving the second term on the LHS to the RHS and using the above claim, we get
\[
(1-\delta)\frac{\v(B(x,(1-\eta)R))}{\v(\underline{B}(0,(1-\eta)R))}\leq (1+\Psi)\frac{\vol(\pa B(x,(1-\eta)R))}{\vol(\pa \underline{B}(0,(1-\eta)R))},
\]
where
\[
\delta=\delta'+\left[2^n\Psi(1+\Psi)+\delta'\right]\frac{\v(\underline{B}(0,(1-\eta)R))}{\vol(\underline{A}_{(1-\eta)R,R}(0))}.
\]
This finishes the proof of the lemma.
\end{proof}

\section*{Acknowledgements}  Q.S.Z is grateful to the Simons foundation for its support. M. Z would like to thank Prof. Huai-Dong Cao for constant support and his interest in this paper. M. Z's research is partially supported by NSFC Grant No. 11501206.
Both authors wish to thank Professors G.F. Wei,   Z.L. Zhang  and X.H. Zhu  for helpful suggestions.

\end{document}